\documentclass[sn-mathphys]{sn-jnl}

\usepackage{amsmath}
\usepackage{amssymb}
\usepackage{amsfonts}
\usepackage{amsthm}
\usepackage{mathtools}
\usepackage{enumerate}
\usepackage{marginnote}
\usepackage{tikz}
\usetikzlibrary{patterns}
\def\triplet{}

\newcommand{\eps}{\varepsilon}

\newcommand{\N}{{\mathbb{N}}}

\newcommand{\R}{{\mathbb{R}}}
\newcommand{\C}{{\mathbb{C}}}

\newcommand{\LL}{\mathrm L}
\newcommand{\HH}{\mathrm H}
\newcommand{\Ce}{\mathrm C}
\newcommand{\We}{\mathrm W}

\newcommand{\fA}{\mathfrak{A}}
\newcommand{\fB}{\mathfrak{B}}
\newcommand{\fC}{\mathfrak{C}}
\newcommand{\fD}{\mathfrak{D}}
\newcommand{\id}{\mathrm I}
\newcommand{\e}{\mathrm e}
\newcommand{\Hlr}[1]{\mathrm H^1_{#1}(0,1)}

\let\Re\relax

\DeclareMathOperator{\Re}{Re}

\DeclareMathOperator{\dom}{dom}

\newcommand{\ddts}{\tfrac{\text{\normalfont d}}{\text{\normalfont d}t}}
\newcommand{\ds}[1]{{\rm \, d} #1 \,}
\newcommand{\setdef}[2]{\left\{ #1 \left\vert\vphantom{#1}\, #2 \right.\right\}}

\DeclareMathOperator{\ext}{ext}
\DeclareMathOperator{\grad}{grad}
\DeclareMathOperator{\Div}{div}

\newcommand{\bvek}[2]
{
   \begin{bmatrix}
      #1\\
      #2
   \end{bmatrix}
}

\newcommand{\kvek}[3]{\left[ \begin{smallmatrix}  #1_{#3} \\ #2_{#3}\end{smallmatrix}\right]}
\newcommand{\smbvek}[2]{\mbox{\footnotesize$\left[\begin{smallmatrix}#1\\#2\end{smallmatrix}\right]$}}
\newcommand{\sembvek}[3]{\mbox{\footnotesize$\left[\begin{smallmatrix}#1\\#2\\#3\end{smallmatrix}\right]$}}
\newcommand{\sbvek}[2]{{\left[\begin{smallmatrix}#1\\#2\end{smallmatrix}\right]}}
\newcommand{\sebvek}[3]{\left[\begin{smallmatrix}#1\\#2\\#3\end{smallmatrix}\right]}
\newcommand{\spvek}[2]{\left(\begin{smallmatrix}#1\\#2\end{smallmatrix}\right)}

\newcommand{\sbmat}[4]{\left[\begin{smallmatrix}#1 & #2\\#3 & #4\end{smallmatrix}\right]}

\newcommand{\widebar}[1]{\overline{#1}}

\jyear{2022}%

\theoremstyle{thmstyleone}%
\newtheorem{theorem}{Theorem}[section]%
\newtheorem{proposition}[theorem]{Proposition}%
\newtheorem{lemma}[theorem]{Lemma}%
\newtheorem{assumptions}[theorem]{Assumption}%

\theoremstyle{thmstyletwo}%
\newtheorem{remark}[theorem]{Remark}%

\theoremstyle{thmstylethree}%
\newtheorem{definition}[theorem]{Definition}%

\begin{document}

\title[Dynamic iteration for dissipative systems]{Operator splitting based dynamic iteration for  linear infinite-dimensional port-Hamiltonian systems}

\author[1]{\fnm{Bálint} \sur{Farkas}}\email{farkas@uni-wuppertal.de}
\author*[1]{\fnm{Birgit} \sur{Jacob}}\email{bjacob@uni-wuppertal.de}

\author[2]{\fnm{Timo} \sur{Reis}}\email{timo.reis@tu-ilmenau.de}

\author[1]{\fnm{Merlin} \sur{Schmitz}}\email{meschmitz@uni-wuppertal.de}
\affil*[1]{\orgdiv{IMACM}, \orgname{Bergische Universit\"at Wuppertal}, \orgaddress{\street{Gau\ss stra\ss e 20}, \city{Wuppertal}, \postcode{D-42119}, \country{Germany}}}

\affil[2]{\orgdiv{Institut f\"ur Mathematik}, \orgname{Technische Universit\"at Ilmenau}, \orgaddress{\street{Weimarer Str.\ 25}, \city{Ilmenau}, \postcode{98693 }, \country{Germany}}}

\abstract{A dynamic iteration scheme for linear infinite-dimensional port-Hamiltonian systems  is proposed. 
The dynamic iteration is monotone in the sense that the error is decreasing, it does not require any  stability condition and is in particular applicable to port-Hamiltonian formulations arising from domain decompositions. }



\maketitle

\section{Introduction}\label{sec:intro}
Operator splitting methods are widely used to reduce (the numerical) solution of a complex problem to the iterative solution of subproblems, into which the original problem is split. 
How this splitting arises can be based on various considerations, coming from the governing physical laws, from the geometry of the domain over which a certain partial differential equation is considered, from the structure of the problem, from mathematical reasons, or from the combination of these. See, for instance, \cite{Marchuk}, \cite{Quispel}, \cite{GeiserBook}, \cite[Ch.{} IV]{HunVerBook}, \cite{HairerGeom}, \cite{Strang68}, \cite{JahnkeLubich}, \cite{HanO09}, \cite{Faou15}, \cite{HochbruckOstermann},  \cite{Hansen1}, \cite{Hansen3} for more information and  a general overview of splitting methods in various situations. 

\medskip \emph {Operator splitting} is particularly suitable for problems composed of subsystems which are coupled in a particular way; although such a coupling may not be visible immediately. For example, boundary coupled system or problems with dynamic boundary condition, see, e.g., \cite{CsEF}, \cite{CsFK}, as well as delay equations have been successfully treated via splitting methods, see, e.g.,   \cite{BelZeBook}, \cite{BelZeActa} or \cite{BCsF2} for an operator splitting approach in the dissipative situation. It is therefore interesting and important to study operator splitting methods for a class of systems that is closed under certain types of couplings. In this paper we are interested in infinite dimensional port-Hamiltonian systems of  a specific structure.

The splitting method studied in this paper is originally due to Peaceman and Rachford, see \cite{PeacemanRachford}, who introduced it in the setting of linear operators. The Peaceman--Rachford splitting was then extended to maximal monotone operators on Banach spaces by Lions and Mercier \cite{Lions1979}. And this framework is indeed more suitable for the purposes of this paper, as the occurring operators here will be only affine linear in general. For error analyis of Peaceman--Rachford type splittings and variants we refer, e.g., to \cite{Hansen2}.

Operator splitting based dynamic iteration schemes for finite dimensional port-Hamiltonian systems were studied in \cite{BaGuJaRe21}. Here we make the first steps to extend the study to infinite dimensional port-Hamiltonian systems. The Peaceman--Rachford--Lions--Mercier type splitting algorithm will result in a convergent approximation, under suitable conditions, but most importantly the approximation error will be monotonically increasingly convergent to $0$, a feature that is connected with the port-Hamiltonian structure of the problem.

This paper is structured as followes. In the next section we give an explicit description of the dynamic iteration scheme and present our main theorems. The proof of these theorems can be found in Section \ref{sec:proof}. In Section \ref{sec:sysnodes} we give some background information on the theory of system nodes that will be used in the proof. Further, Section \ref{sec:maxmon} is devoted to the proof of maximal monotonicity of one of the splitting operators. 
We end in Section \ref{sec:ex} with two examples: A coupled wave-heat system and a domain decomposition for the wave equation.

\section{Description of the dynamic iteration scheme}\label{sec:dynit}

We consider $n\in \mathbb N$ linear (infinite-dimensional) port-Hamiltonian systems 
\begin{align}
\begin{bmatrix}
\dot{x}_i(t)\\ \mathfrak y_i(t)\\  y_i(t)
\end{bmatrix}
&= S_i \begin{bmatrix}
x_i(t)\\ \mathfrak u_i(t) \\  u_i(t)
\end{bmatrix},\qquad t\ge 0, \quad i=1,\ldots,n,\label{eq:singlesystem1}\\
x_i(0) &= x_{i0}, \qquad i=1,\ldots,n,\nonumber
\end{align}
where $x_i(t)\in X_i$ denotes the state, $\mathfrak u_i(t)\in \mathfrak U_i$ and $ u_i(t)\in  U_i$ denote inputs and $\mathfrak y_i(t)\in \mathfrak U_i$ and $ y_i(t)\in  U_i$ denote outputs of system $i$ at time $t$. Here $\mathfrak U_i$, $ U_i$ and $X_i$ are Hilbert spaces. Denoting the Cartesian product of the spaces $Z_1,\ldots,Z_n$ by $\sebvek{Z_1}\vdots{Z_n}$, we define
\begin{align*}
X \coloneqq\sebvek{X_1}\vdots{X_n}, \quad
 \mathfrak U \coloneqq  \sebvek{{\mathfrak U}_1}{\vdots}{{\mathfrak U}_n}\,\text{ and }\,
 U \coloneqq \sebvek{U_1}\vdots{U_n}.
\end{align*}
We assume that   the linear operators $S_i$, $i=1,\ldots, n$, are system nodes on the Hilbert space triples $(\sbvek {\mathfrak U_i}{U_i},X_i,\sbvek {\mathfrak U_i}{U_i})$. We recall the definition of a system node in Section \ref{sec:sysnodes}. In particular, this class covers well-posed linear systems \cite{St05}, boundary control and observation systems \cite{TuWe09} and, of course, linear infinite-dimensional systems with bounded control and observation   \cite{CuZw20}.

Further, we assume that the systems are coupled via
\begin{equation}  \label{eq:coupling} 
 \begin{bmatrix}  y_1(t)\\\vdots \\  y_n(t)\end{bmatrix} =N_c\begin{bmatrix}  u_1(t)\\\vdots \\  u_n(t)\end{bmatrix} ,\qquad t\ge 0,
\end{equation}
where $N_c$ is a bounded linear operator from $U$ to $ U$ satisfying $\Re \langle y,N_c y\rangle \le 0$ for every $y\in  U$. Hence, the fraktur typeface indicates that the functions can be interpreted as external inputs and outputs that build  inputs and outputs of the closed system. 

As the systems  \eqref{eq:singlesystem1}  are port-Hamiltonian and the coupling operator $N_c$ satisfies $\Re \langle y,N_c y\rangle \le 0$ for every $y\in  U$, the interconnected system is again a port-Hamiltonian system. We note, that the systems \eqref{eq:singlesystem1} are port-Hamiltonian if and only if the corresponding system nodes $S_i$ are impedance passive. Further, as stated in \cite[Thm. 4.2]{St02a} a system node is impedance passive if and only if it is (maximal) dissipative.

The aim of this article is to develop for given inputs $\mathfrak u_1, \ldots, \mathfrak u_n$ and given initial conditions $x_{10}, \ldots,x_{n0}$ for the closed loop system  \eqref{eq:singlesystem1}-\eqref{eq:coupling} a   
dynamic iteration scheme which allows to solve the linear port-Hamiltonian systems $S_i$ separately and also in parallel.

Every system node $S_i$ on $(\sbvek {\mathfrak U_i}{U_i}, X_i, \sbvek{\mathfrak U_i}{U_i} )$ can be written as
\[
S_i = \left[ \begin{smallmatrix}
   \phantom{[} A_i\&  B_i\phantom{]_c}\\ [C_i\& D_i]_1\\ [C_i \& D_i]_2
\end{smallmatrix}\right].
\]
Here $A_i\&  B_i\coloneqq  P_{X_i} S_i$, $[C_i\& D_i]_1\coloneqq P_{\mathfrak U_i}S_i$ and $[C_i\& D_i]_2\coloneqq P_{U_i}S_i$, where $P_{X_i}$, $P_{\mathfrak U_i}$ and $P_{U_i}$ are the canonical projections onto ${X_i}$, ${\mathfrak U_i}$ and ${U_i}$ in $X_i \times {\mathfrak U_i} \times U_i$.

Let $S$ be the system node on $(\sbvek {\mathfrak U}{U},X,\sbvek{\mathfrak U}{U})$ with $S_i$ "on the diagonal", i.e. the operator $S$ is of the form 
\[
S = \left[ \begin{smallmatrix}
   \phantom{[} A \& B \phantom{]_1}\\ [C\& D]_1 \\ [C\& D]_2
\end{smallmatrix} \right],
\]
Our assumptions on the system read as follows

\begin{assumptions}[on the node]\label{ass:disssyst}
Let $X$, $U$ be Hilbert spaces. The linear operator $S = \left[ \begin{smallmatrix}
   \phantom{[} A \& B \phantom{]_1}\\ [C\& D]_1 \\ [C\& D]_2
\end{smallmatrix} \right] \colon \dom(S)\subset \sebvek{X}{\mathfrak U}{U}\to \sebvek{X}{\mathfrak U}{U}$ (with $A\&B=P_XS$, $ \sbvek{{[C \& D]_1}}{{[C \& D]_2}}=P_{\sbvek{\mathfrak U} {U}}S$) has the following properties:
\begin{enumerate}[(i)]
\item $\left[ \begin{smallmatrix}
   \phantom{-[} A \& B \phantom{]_1}\\ -[C\& D]_1\\ -[C\& D]_2
\end{smallmatrix} \right]$ is dissipative,
\item $S$ is closed. Further, $A\&B$ is closed with $\dom(A\&B)=\dom(S)$.
\item For all $\sbvek{\mathfrak u}{u}\in \sbvek{\mathfrak U}  U$, there exists some $x\in X$ with $\sebvek x{\mathfrak u}{u}\in\dom(A\&B)$.
\item The main operator $A\colon \dom(A)\subset X\to X$ with \[\dom(A) \coloneqq \setdef{x\in X}{(x,0,0)\in\dom(S)}\] and $A{x} \coloneqq P_ X S\left(\begin{smallmatrix}
    x\\0\\0
\end{smallmatrix}\right)$ for all $x\in\dom(A)$ fulfills
\[\rho(A)\cap\C_+\neq\emptyset.\]
Here $\rho(A)$ denotes the resolvent set of the linear operator $A$, and $\C_+\coloneqq \setdef{\lambda\in\C}{\Re \lambda>0}$.
\end{enumerate}
\end{assumptions}
We abbreviate the solution space $H \coloneqq \LL^2 ([0,T]; \sbvek X U)$.
For fixed $T>0$,  a function $\mathfrak{u}\colon [0,T]\to \mathfrak U$ and  $x_0\in X$ we consider the operator
\begin{subequations}\label{eq:Mop2}
\begin{equation}
M\colon \dom(M)\subset H\to H
\end{equation}
with
\begin{equation}
\dom(M)=\setdef{\left[\begin{smallmatrix}
    x \\ u
\end{smallmatrix}\right]\in H}{\parbox[c]{6.2cm}{$\displaystyle{\left[\begin{smallmatrix}
    \dot x\\ 0 
\end{smallmatrix}\right]-\left[ \begin{smallmatrix}
   \phantom{-[} A \& B \phantom{]_1}\\ -[C\& D]_2 
\end{smallmatrix} \right] \left[ \begin{smallmatrix}
    x \\ \mathfrak u \\ u
\end{smallmatrix} \right] \in H} $ and $x(0)=x_0$}},
\end{equation}
\begin{equation}
M \begin{bmatrix}
    x \\ u
\end{bmatrix}=\begin{bmatrix}
    \dot{x}-A\&B  \left[ \begin{smallmatrix}
    x \\ \mathfrak u \\ u
\end{smallmatrix} \right]\\
[C\&D]_2  \left[ \begin{smallmatrix}
    x \\ \mathfrak u \\ u
\end{smallmatrix} \right]
\end{bmatrix}.
\end{equation}
\end{subequations}
The precise meaning of $\dot x$ will be clarified in Section \ref{sec:sysnodes}, when we discuss system nodes, and solution trajectories, see also Remark \ref{rem:sol}.
Note that $M$ is not a linear operator unless $x_0=0$ and $\mathfrak u=0$, since it is in general not defined on a~vector space.
Further,  we define
$N\in \mathcal L(\LL^2([0,T];\sbvek XU))$ by 
\begin{align}
\label{def:N} N\begin{bmatrix} x \\u \end{bmatrix}\coloneqq \begin{bmatrix}0\\  -N_c  u \end{bmatrix}.
\end{align}

We assume that the coupling is such that $N$ is a maximal monotone operator.
Thus for $\lambda>0$ the operator $(\id - \lambda N)(\id + \lambda N)^{-1}$ is a contraction, see Section \ref{sec:maxmon}.
\begin{remark}
If we consider two systems ($n = 2$) the standard negative feedback $ u_1= y_2$, $ u_2=- y_1$ yields a coupling matrix $N_c = \left[\begin{smallmatrix*}[r]
0&-\id \\ \id&0 
\end{smallmatrix*}\right]$.
\end{remark}

The  system arising from the coupling of  $S_i$, $i=1,\ldots,n$ via $N_c$ (without the output equation for $\mathfrak y$) is equivalent to the equation
\begin{align}\label{eq:algsplitting}
M\left[\begin{smallmatrix}
    x \\ u
\end{smallmatrix}\right] +N\left[\begin{smallmatrix}
    x \\ u
\end{smallmatrix}\right]=0
\end{align}
(see Remark \ref{rem:sol}), which is equivalent to 
\[  \left[\begin{smallmatrix}
    x \\ u
\end{smallmatrix}\right] = (\id +\lambda M)^{-1}(\id - \lambda N)(\id + \lambda N)^{-1}(\id - \lambda M)\left[\begin{smallmatrix}
    x \\ u
\end{smallmatrix}\right],\]
where $\lambda >0$ is arbitrary, see Section \ref{sec:maxmon} for the discussion of the inverse mappings appearing here.
We consider an algorithm inspired by ideas of Lions and Mercier as  in \cite{Lions1979}: 
\begin{align}\label{def:algorithm}
    \sbvek{x_{k+1}}{u_{k+1}} = (\id +\lambda M)^{-1}(\id - \lambda N)(\id + \lambda N)^{-1}(\id - \lambda M)\sbvek{x_k}{u_k},
\end{align}
with $\sbvek{x_0}{u_0} \in \dom(M)$ arbitrary.
Now we can formulate our second assumption:
\begin{assumptions}[Solution]\label{ass:solution}
For fixed $x_0 \in X$, $T>0$ and $\mathfrak u \in \LL^2([0,T],\mathfrak U)$ there exists a solution $\sbvek xu$ to the equation \eqref{eq:algsplitting} on $[0,T]$.
\end{assumptions}
The main results of this paper are the following:
\begin{theorem}\label{thm:convergence}
  Let Assumptions \ref{ass:disssyst}, \ref{ass:solution} be fulfilled.
  For the operators $M,N$ as defined in \eqref{eq:Mop2} and \eqref{def:N} let the sequence $\sbvek {x_k}{u_k}_{k}$ be defined by \eqref{def:algorithm}.
 Then:
  \begin{enumerate}[a)]
    \item For the sequence $\sbvek{w_k}{z_k}_k$ defined by
    \[ \sbvek{w_k}{z_k} \coloneqq (\id + \lambda M)\sbvek{x_k}{u_k}, \quad k\in \N,\]
    and the function $\sbvek{w}{z} \coloneqq (\id + \lambda M)\sbvek{x}{u}$, the sequence $(\| \sbvek{w_k}{z_k} - \sbvek{w}{z}\|_{2})_k$ is monotonically decreasing and 
    \[ \| \sbvek{x_k}{u_k} - \sbvek{x}{u} \|_{2} \le \|\sbvek{w_k}{z_k} - \sbvek{w}{z}\|_{2}, \quad \forall k \in \N . \]
    \item $(x_{k})_{k}$ converges to $x$ in $\LL^{2}([0,T]; X )$.
    \item $(x_{k})_{k}$ converges uniformly to $x$ on $[0,T]$.
  \end{enumerate}
\end{theorem}

\begin{theorem}\label{thm:outputconvergence}
Let additionally to the assumptions of Theorem \ref{thm:convergence} the system be partially strictly output passive with regard to the external output, i.e. 
there is $\eps >0$ such that for all $\sebvek {x}{\mathfrak u}{u} \in \dom(S)$ the inequality
\begin{align*}
   \emph{(PSOP)} \qquad \Re \left\langle \sebvek {\phantom{-[}A \& B\phantom{]_2}}{{-[C\&D]_1}}{{-[C\&D]_2}}\sebvek x{\mathfrak u}u , \sebvek x{\mathfrak u}u \right\rangle_{\sembvek X{\mathfrak U}U} \le -\eps \left\|  [C\&D]_1 \sebvek {x}{\mathfrak u}{u} \right\|^2_{\mathfrak U}
\end{align*}
 holds. Then, if $T$, $x_0$, $\mathfrak{u}$, $x$, $u$ are given as in Theorem \ref{thm:convergence}, the corresponding external output also converges to  $\mathfrak y \coloneqq [C\&D]_1 \sebvek {x}{\mathfrak u}{u}$, i.e.
 \begin{align*}
   \left\| [C\&D]_1 \sebvek {x_k}{\mathfrak u}{u_k} -\mathfrak y \right\|_{2,\mathfrak U} \longrightarrow 0.
 \end{align*}
\end{theorem}

\begin{remark}
Using the same argument, we obtain convergence of the internal outputs (i.e., $\lim_{k\to \infty}\|y - y_k\|_{2,U} =0$) under the assumption of partial strict output passivity with regard to the internal output. 
Then, using \eqref{eq:coupling} and the boundedness of $N_c$ we easily see the convergence of the internal inputs $u_k$.\\
Since we can assume invertibility of $N_c$ without loss of generality, the same holds if the system is partially strictly input passive with regard to the internal input.
\end{remark}

\begin{remark}
\begin{enumerate}
    \item The block structure of $S$ allows a parallelized computation of the subsystems $S_i$. 
    \item In the splitting algorithm \eqref{def:algorithm} for the sum of two operators such as \eqref{eq:algsplitting} one can interpret the variable $\lambda$ as a time step. Therefore this algorithm represents a combination of steps for the first operator alternating ones for the second.
\end{enumerate}
\end{remark}

\section{Background on system nodes}\label{sec:sysnodes}

Let $X$, $U$ and $Y$ be Hilbert spaces and denote the canonical projections onto $X$ and $Y$ in $\sbvek XU$ respectively by $P_X$
 and $P_Y$. Let
\[
S \colon \dom(S)\subset \sbvek XU \to\sbvek XY
\]
be a linear operator. Its corresponding {\em main operator} is given by $A \colon \dom(A)\subset X\to X$ with $\dom(A) \coloneqq \setdef{x\in X}{\sbvek x0\in\dom(S)}$ and $Ax \coloneqq P_ X S\sbvek x0$ for all $x\in\dom(A)$. One often sets
\[
A\& B \coloneqq P_X S
\qquad\text{and}\qquad
C\& D \coloneqq P_Y S
\]
so  $S$ can be written as
\[
S = \bvek{A\& B}{C\& D}.
\]
The concept of system nodes poses natural assumptions on the operator $S$, in order to guarantee favorable properties and a suitable solution concept to the dynamics specified by the differential equation
\begin{equation}
\begin{bmatrix}
\dot{x}(t)\\y(t)
\end{bmatrix}
= S \begin{bmatrix}
x(t)\\u(t)
\end{bmatrix}.\label{eq:ODEnode}
\end{equation}
For a comprehensive study of system nodes, we refer to the monograph \cite{St05}.

\begin{definition}[System node]
A {\em system node} on the triple $\triplet(Y,X,U)$ of Hilbert spaces is a (possibly unbounded) linear operator $S \colon \dom(S)\subset \sbvek XU\to \sbvek XY$ satisfying the following conditions:
\begin{enumerate}[(i)]
\item $S$ is closed.
\item $P_X S\colon \dom(S)\subset \sbvek XU\to X$ is closed.
\item For all $u\in U$, there exists some $x\in X$ with $\sbvek{x}{u}\in \dom(S)$.
\item The main operator $A$ is the generator of a~strongly continuous semigroup $\fA(\cdot)\colon [0,\infty)\to {\mathcal L}(X)$ on $X$.
\end{enumerate}
\end{definition}

\begin{remark}[System nodes]\label{rem:nodes}
Let $S = \sbvek{A\& B}{C\& D}$ be a~system node on $\triplet(Y,X,U)$.
\begin{enumerate}[(i)]
\item\label{rem:nodes1} It follows from the above definition that $C\&D\in \mathcal L(\dom(A\&B),Y)$, where $\dom(A\&B)$ is endowed with the graph norm of $A\&B$. In particular, the operator $C$ with $Cx\coloneqq C\&D\sbvek x0$ fulfills $C\in \mathcal L(\dom(A),Y)$.
\item\label{rem:nodes3}  Since generators of semigroups are densely defined (see \cite[Chap.~2, Thm.~1.5]{EnNa00}), $\dom(S)$ is dense in $\sbvek XU$ and for given $u\in U$ the affine subspace
\[\setdef{x\in X}{\sbvek xu\in \dom(A\& B)}\]
is dense in $X$.
\item\label{rem:nodes4} Since $A$ is a generator of a $C_0$-semigroup there is $\alpha\in\rho(A)$ (in the resolvent set of $A$). The completion of $X$ with respect to the norm $\|x\|_{X_{-1}}\coloneqq \|(\alpha \id-A)^{-1}x\|$ is denoted by $X_{-1}$.   Note that the topology of $X_{-1}$ does not depend on the particular choice of $\alpha\in \rho(A)$ \cite[Prop.~2.10.2]{TuWe09}. The operator $A$ extends continuously as $A_{-1}\colon X\mapsto X_{-1}$; $A$ and $A_{-1}$ are similar, hence have the same spectrum and $A_{-1}$ generates a $C_0$-semigroup $\fA_{-1}(\cdot)\colon [0,\infty)\to {\mathcal L}(X_{-1})$ on $X_{-1}$, which extends $\fA(\cdot)$ (and which is similar to $\fA(\cdot)$), see \cite[Sec.{} II.5]{EnNa00}.

\item\label{rem:nodes2} $A\& B$ extends to a bounded linear operator $[A_{-1}\ B]\colon\sbvek XU\to X_{-1}$, which in fact has such a block structure. Moreover, the domain of $A\&B$ (equally: the domain of $S$) fulfills
    \[\dom(A\&B)=\setdef{\sbvek xu \in \sbvek X U}{A_{-1}x+Bu\in X},\]
    see \cite[pp.{} 3--4]{St05}.

\item\label{rem:nodes5} For all $\alpha\in\rho(A)$ the norm
\[\left\|\spvek xu\right\|_\alpha\coloneqq \left(\|x-(\alpha \id-A_{-1})^{-1}Bu\|_X^2+\|u\|^2_U\right)^{1/2}\]
is equivalent to the graph norm of $S$. Moreover, the operator
\[\sbmat \id{-(\alpha \id -A_{-1})^{-1}B}0\id\]
maps $\dom(S)$ bijectively to $\sbvek{\dom(A)}U$, see \cite[Lem.~4.7.3]{St05}.
\end{enumerate}
\end{remark}

Remark~\ref{rem:nodes}\,\eqref{rem:nodes5} allows to define the concept of the transfer function.
\begin{definition}[Transfer function]
Let $S = \sbvek{A\& B}{C\& D}$ be a~system node. The \emph{transfer function associated to $S$} is
\[\begin{aligned}
\widehat{\mathfrak D}\colon &&\rho(A)\to &\,\mathcal L(U,Y),\\&&s\mapsto&\,C\& D \bvek{(s\id-A_{-1})^{-1}B}\id.
\end{aligned}\]
\end{definition}
Next we briefly recall suitable solution concepts for the differential equation
\eqref{eq:ODEnode} with $S=\sbvek{A\&B}{C\&D}$ being a system node.

\begin{definition}[Classical/generalized trajectories]\label{def:traj}
Let $S = \sbvek{A\& B}{C\& D}$ be a~system node  on $\triplet(Y,X,U)$, and let $T>0$.\\
A {\em classical trajectory} for \eqref{eq:ODEnode} on $[0,T]$ is a triple
\[
(x,u,y)\,\in\,\Ce^1([0,T];X)\times \Ce([0,T];U)\times \Ce([0,T];Y)
\]
which for all $t\in[0,T]$ satisfies \eqref{eq:ODEnode}.\\
A {\em generalized trajectory} for \eqref{eq:ODEnode} on $[0,T]$ is a triple
\[
(x,u,y)\,\in\,\Ce([0,T];X)\times \LL^2([0,T];U)\times \LL^2([0,T];Y),
\]
which is a~limit of classical trajectories for \eqref{eq:ODEnode} on $[0,T]$ in the topology of $\Ce([0,T];X)\times \LL^2([0,T];U)\times \LL^2([0,T];Y)$.
\end{definition}
If $S = \sbvek{A\& B}{C\& D}$ is a~system node  on $\triplet(Y,X,U)$, then $A\&B$ can be regarded as a~system node on $\triplet(\{0\},X,U)$. Consequently, we may further speak of classical (generalized) trajectories $(x,u)$ for $\dot{x}=A\&B\sbvek xu$.

The following result ensures the existence of unique classical trajectories with suitable control functions and initial values.
\begin{proposition}[Existence of classical trajectories {\cite[Thm.~4.3.9]{St05}}]\label{prop:solex}
Let $S$ be a system node on $\triplet(Y,X,U)$, let $T>0$, $x_0\in X$ and $u\in \We^{2,1}([0,T];U)$ with $\sbvek{x_0}{u(0)}\in \dom(S)$. Then  there exist unique
classical trajectory $(x,u,y)$ for \eqref{eq:ODEnode} with $x(0)=x_0$.
In the case of a well-posed system $u \in \We^{1,2}([0,T];U)$ is sufficient for the existence of classical trajectories and one also gets $y \in \We^{1,2}([0,T];Y)$ (see \cite[p. 298]{St02a}).
\end{proposition}

We provide some further statements on classical/generalized trajectories.
\begin{remark}[Classical/generalized trajectories]\label{rem:sols}
Let $S = \sbvek{A\& B}{C\& D}$ be a~system node  on $(Y,X,U)$, and let $T>0$.
\begin{enumerate}[(i)]
\item\label{rem:sols1} Assume that $(x,u)$ is a~classical trajectory for $\dot{x}=A\&B\sbvek xu$. Then $\sbvek xu\in \Ce([0,T];\dom(S))$.
\item\label{rem:sols2} $(x,u)$ is a~generalized trajectory for $\dot{x}=A\&B\sbvek xu$ if, and only if,
\begin{equation}\forall\,t\in[0,T]:\quad x(t)=\fA(t)x(0)+\int_0^t \fA_{-1}(t-\tau)Bu(\tau)\,{\rm d}\tau,\label{eq:mildsol}\end{equation}
where the latter has to be interpreted as an integral in the space $X_{-1}$. In particular, $x\in \Ce([0,T];X_{-1})$.
\item\label{rem:sols3} If $(x,u,y)$ is a~generalized trajectory for \eqref{eq:ODEnode}, then, clearly, $(x,u)$ is a~generalized trajectory for $\dot{x}=A\&B\sbvek xu$. In particular, \eqref{eq:mildsol} holds. The output evaluation $y(t)=C\&D \sbvek {x(t)}{u(t)}$ is---at a first glance---not necessarily well-defined for all $t\in[0,T]$. However, it is shown in \cite[Lem.~4.7.9]{St05} that the second integral of $\sbvek {x}{u}$ is continuous as a~mapping from $[0,T]$ to $\dom(A\&B)=\dom(S)$. As a~consequence, the output can---in the distributional sense---be defined as the second derivative of $C\&D$ applied to the second integral of $\sbvek {x}{u}$. This can be used to show that $(x,u,y)$ is a~generalized trajectory for \eqref{eq:ODEnode} if, and only if, $(x,u)$ is a~generalized trajectory for $\dot{x}=A\&B\sbvek xu$, and
    \[y=\left(t\mapsto\tfrac{{\rm d}^2}{{\rm d}t^2}\,C\&D\int_0^t(t-\tau)\sbvek {x(\tau)}{u(\tau)}\,{\rm d}\tau\right)\in \LL^2([0,T];Y).\]
\end{enumerate}
\end{remark}
Next we recall the important concept of well-posed systems.
\begin{definition}[Well-posed systems]\label{def:wp}
Let $S = \sbvek{A\& B}{C\& D}$ be a~system node  on $(Y,X,U)$. The system \eqref{eq:ODEnode} is called \emph{well-posed}, if for some (and hence all) $T>0$, there exists some $c_T>0$, such that the classical (and thus also the generalized) trajectories for \eqref{eq:ODEnode} on $[0,T]$ fulfill
\[\|x(t)\|_X+\|y\|_{\LL^2([0,T];Y)}\leq c_T\big(\|x(0)\|_X+\|u\|_{\LL^2([0,T];U)}\big).\]
\end{definition}
\begin{remark}[Well-posed systems]\label{rem:wp}
Let $S = \sbvek{A\& B}{C\& D}$ be a~system node on $\triplet(Y,X,U)$ and $T>0$. Well-posedness of \eqref{eq:ODEnode} is equivalent to boundedness of the mappings
\[\begin{aligned}
\fB_T\colon &&\LL^2([0,T];U)\to&\, X,\quad&\fC_T\colon && X\to&\, \LL^2([0,T];Y),\\
\fD_T\colon &&\LL^2([0,T];U)\to&\, \LL^2([0,T];Y),
\end{aligned}\]
where
\begin{itemize}
\item $\fB_T u=x(T)$, where $(x,u,y)$ is the generalized trajectory for \eqref{eq:ODEnode} on $[0,T]$ with $x(0)=0$,
\item $\fC_T x_0=y$, where $(x,u,y)$ is the generalized trajectory for \eqref{eq:ODEnode} on $[0,T]$ with $u=0$ and $x(0)=x_0$,
\item $\fD_T u=y$, where $(x,u,y)$ is the generalized trajectory for \eqref{eq:ODEnode} on $[0,T]$ with $x(0)=0$.
\end{itemize}
In view of Remark~\ref{rem:sols}\,\eqref{rem:sols2}, we have
\[\fB_T u=\int_0^t \fA_{-1}(t-\tau)Bu(\tau)\,{\rm d}\tau\quad \forall u\in \LL^2([0,T];U).\]
In particular, well-posedness implies that the above integral is an element of $X$.
Since the domain of the generator of a $C_0$-semigroup is invariant under the semigroup operators, for each $t>0$ and $x_0\in\dom(A)$ one has $\fA(t)x_0\in\dom(A)$.
Thus, with $C$ as in Remark~\ref{rem:nodes}\,\eqref{rem:nodes1}, we have that for $y=C\fA(\cdot)x_0$, $x=\fA(\cdot)x_0$, $(x,0,y)$ is a~classical trajectory  for \eqref{eq:ODEnode} on $[0,T]$ with $x(0)=x_0$. Well-posedness implies that the mapping $x_0\mapsto C\fA(\cdot)x_0$ has an extension to a~bounded linear operator $\fC_T\colon X\to\LL^2([0,T];Y)$, see \cite[Thm.~4.7.14]{St05}.
\end{remark}

\begin{lemma}\label{lem:systext}
Let $S = \sbvek{A\& B}{C\& D}$ be a~system node on $\triplet(Y,X,U)$. Then
\begin{equation}
S_{\ext}=\begin{bmatrix}A \& B & \id\\C\& D& 0\\\id\phantom{\&}0 &0\end{bmatrix}
\end{equation}
is a~system node on $\triplet(\sbvek YX,X,\sbvek UX)$. Further, if \eqref{eq:ODEnode} is well-posed, then
\begin{equation}
\begin{bmatrix}
\dot{x}(t)\\y_{\ext}(t)
\end{bmatrix}
= S_{\ext} \begin{bmatrix}
x(t)\\u_{\ext}(t)
\end{bmatrix}\label{eq:ODEnodeext}
\end{equation}
is well-posed.
\end{lemma}
It is straightforward to verify that $S_{\ext}$ is a~system node.
The proof of the equivalence between well-posedness of \eqref{eq:ODEnode} and \eqref{eq:ODEnodeext} consists of a~straightforward combination of Remark~\ref{rem:wp} with \cite[Thm.~4.4.4\&4.4.8]{St05} and is therefore omitted.

Next we recap the notion of partial flow inverse from \cite[Def.~6.6.6]{St05}, which will turn out to be corresponding to a~system in which the second part of input is interchanged with the second part of the output.
\begin{definition}[Partial flow inverse] \label{def:partflow}
A system node $ S = \left[ \begin{smallmatrix}  \phantom{[} A\& B \phantom{]_1} \\ [C\& D]_1 \\ [C \& D]_2 \end{smallmatrix} \right] $ on $\left(\sbvek{\mathfrak Y}{ Y} , X, \sbvek{\mathfrak U}{ U} \right)$ with main operator $A$, control operator $ B = [\mathfrak B \ \widehat B]$ and observation operator $ C = \sbvek{\mathfrak C}{\widehat C}$ is called \emph{partially flow-invertible} if there exists a system node $S^{\curvearrowleft} = \left[ \begin{smallmatrix}    [A\& B]^{\curvearrowleft} \\ [C\& D]_{1}^{\curvearrowleft} \\ [C \& D]_{2}^{\curvearrowleft} \end{smallmatrix} \right] $ on $\left(\sbvek{\mathfrak Y}{ U}, X, \sbvek{\mathfrak U}{ Y} \right)$ satisfying the following condition:
the operator $\left[ \begin{smallmatrix}    1&0&0 \\ 0&1&0 \\ [ C & \& & D]_2 \end{smallmatrix}\right]$ maps $\dom(S)$ continuously onto $\dom(S^{\curvearrowleft})$, its inverse is $\left[ \begin{smallmatrix}    1\ 0 \ 0 \\ 0\ 1\ 0 \\ [ C  \&  D]_2^{\curvearrowleft} \end{smallmatrix}\right]$ and 
\begin{align*}
    S &= \begin{bmatrix}
        [A \& B]^{\curvearrowleft}\\ [C \& D]_1^{\curvearrowleft}\\ 0\ 0 \ 1 
    \end{bmatrix} \begin{bmatrix}
        1\ 0\ 0 \\ 0\ 1\ 0\\ [C \& D]_2^{\curvearrowleft}
    \end{bmatrix}^{-1} && \textrm{on } \dom(S), \\
    S^{\curvearrowleft} &= \begin{bmatrix}
        \phantom{[}A \& B\phantom{]_2}\\ [C \& D]_1\\ 0\ 0 \ 1 
    \end{bmatrix} \begin{bmatrix}
        1\ 0\ 0 \\ 0\ 1\ 0\\ [C \& D]_2
    \end{bmatrix}^{-1} && \textrm{on } \dom(S^{\curvearrowleft}).
\end{align*}
In this case we call $S$ and $S^{\curvearrowleft}$ \emph{partial flow-inverses} of each other.

If $S$ is partially flow invertible then for the transfer function $\widehat{\mathfrak D}= \sbmat {\widehat{\mathfrak D}_{11}}{\widehat{\mathfrak D}_{12}}{\widehat{\mathfrak D}_{21}}{\widehat{\mathfrak D}_{22}}$  of $S$ and for the transfer function $\widehat{\mathfrak D}^\curvearrowleft$ of the system node $S^\curvearrowleft$ 
\[ \widehat{\mathfrak D}_{22}(\alpha) \quad \text{and} \quad  \widehat{\mathfrak D}^\curvearrowleft_{22}(\alpha)\]
are invertible for all $\alpha \in \rho(A)\cap \rho(A^\curvearrowleft)$ and we have $\widehat{\mathfrak D}_{22}^\curvearrowleft (\alpha) = [\widehat{\mathfrak D}_{22}(\alpha)]^{-1}$ (see \cite[Thm. 6.6.9\&6.6.10]{St05}).
\end{definition}

\begin{proposition}[{Partial flow-invertibility, \cite[Thm.~6.6.11]{St05}}]
\label{prop:sysnodepartflowinv}
    A system node  $S = \left[ \begin{smallmatrix}
   \phantom{[} A\& B \phantom{]_1} \\ [C\& D]_1 \\ [C \& D]_2
\end{smallmatrix} \right] $  on $\left( \sbvek{\mathfrak Y}{ Y},X, \sbvek{\mathfrak U}{ U} \right)$ is partially flow-invertible if and only if there exists some $\alpha \in \C$ such that the following two statements are valid:
\begin{enumerate}
    \item The operator $\left[  \begin{smallmatrix}
        \alpha \id& 0&0 \\ 0&\id&0 \\ 0&0&0
    \end{smallmatrix}  \right] - \left[ \begin{smallmatrix}
        A\& B \\ 0 \\ [C \& D]_2
    \end{smallmatrix} \right] $ maps $\dom(S)$ bijectively to $\left[ \begin{smallmatrix}
        X \\ \mathfrak U \\  Y
    \end{smallmatrix} \right]$.
    \item By denoting 
    \[  \left(\left[  \begin{smallmatrix}
        \alpha \id& 0&0 \\ 0&\id &0 \\ 0&0&0
    \end{smallmatrix}  \right] - \left[ \begin{smallmatrix}
        A\& B \\ 0 \\ [C \& D]_2
    \end{smallmatrix} \right]\right)^{-1} = \left[ \begin{smallmatrix}{M_{11}}&{M_{12}}&{M_{13}}\\{M_{21}}&{M_{22}}&{M_{23}}\\{M_{31}}&{M_{32}}&{M_{33}}\end{smallmatrix} \right], \]
    $M_{11}$ is injective, has dense range and $-M_{11}^{-1}$ generates a strongly continuous semigroup.
\end{enumerate}
\end{proposition}

The above definition yields that the partial flow inverse of a partially flow-invertible system node is unique. Although, in the general setting the partial flow inverse of a system node is only an operator node our assumptions of dissipativity guarantee it to be a system node, see Remarks \ref{rem:disssyst} and \ref{rem:flowinvdiss}. Hence, we include this property into our definition.
The main motivation for flow-inverses is that they interchange input and output. This is the subject of the following result, which is a~slight reformulation of \cite[Thm.~6.6.15]{St05}.

\begin{proposition}[{Trajectories and flow-inverse}]\label{prop:trajectories}
    Let $S$ be a flow-invertible system node on $\triplet(\sbvek{\mathfrak Y}{ Y}, X, \sbvek{\mathfrak U}{ U}  )$ whose flow-inverse $S^{\curvearrowleft}$ is also a system node. Then $\left( x, \sbvek{\mathfrak u}{ u}, \sbvek{\mathfrak y}{ y} \right)$ is a classical (generalized) trajectory on $[0,T]$ if and only if $\left( x, \sbvek{\mathfrak u}{ y}, \sbvek{\mathfrak y}{ u} \right)$ is a classical (generalized) trajectory for the system associated to the node $S^{\curvearrowleft}$.
\end{proposition}
As the system node is split up, we also introduce the notation of the subsystems. Let $\Sigma=\sbmat {\Sigma_{11}}{\Sigma_{12}}{\Sigma_{21}}{\Sigma_{22}}$ be the system corresponding to the system node $S$ with transfer function $\widehat{\mathfrak D}= \sbmat {\widehat{\mathfrak D}_{11}}{\widehat{\mathfrak D}_{12}}{\widehat{\mathfrak D}_{21}}{\widehat{\mathfrak D}_{22}}$. This notation will come in handy when talking about the special well-posedness in the next section.

\section{Maximal monotonicity of the operator \texorpdfstring{$M$}{M}}\label{sec:maxmon}

In this section we discuss some important properties of the operator $M$ defined in \eqref{eq:Mop2}.
The most important one is the so-called maximal monotonicity, which is defined as follows.
\begin{definition}
Let $X$ be a Hilbert space with inner product $\langle \cdot, \cdot\rangle$. A set $Y\subset \sbvek X X$ is called {\em monotone}, if
\[
  \Re\langle x-u,y-v \rangle\ge 0, \qquad \sbvek xy,\sbvek uv\in Y.
\]
Further, $Y\subset \sbvek X X$ is called  {\em maximal monotone}, if it is monotone and not a proper subset of a monotone subset of $\sbvek X X$.
A (possibly nonlinear) operator $A\colon \dom(A)\subset X\rightarrow X$ is called {\em (maximal) monotone}, if the graph of $A$, i.e., $\setdef{\sbvek x{Ax}}{x\in \dom(A)}$, is (maximal) monotone.\\
A set $Y\subset \sbvek X X$ is called {\em dissipative}, if
\[
  \Re\langle x,y \rangle\le 0, \qquad \sbvek xy\in Y.
\]
Further, $Y\subset \sbvek X X$ is called  {\em maximal dissipative}, if it is dissipative and not a proper subset of a dissipative subset of $ \sbvek X X$.
A (possibly nonlinear) operator $A\colon \dom(A)\subset X\rightarrow X$ is called {\em (maximal) dissipative}, if the graph of $A$, i.e., $\setdef{\sbvek x{Ax}}{x\in \dom(A)}$, is (maximal) dissipative.
\hfill $\Box$
\end{definition}
Here are some simple implications and equivalences of this property.

\begin{remark}\label{rem:inv}
Let $A\colon \dom(A)\subset X\rightarrow X$ be an operator.
\begin{enumerate}[(i)]
\item If $A$ is linear then $A$ is (maximal) dissipative if, and only if, $-A$ is (maximal) monotone.
\item Assume that $A$ is monotone.
It follows from the definition of monotonicity that $\id+\lambda A$ is injective for all $\lambda>0$.\\
Moreover, the Minty--Browder theory yields the equivalence of the following three statements, see, e.g., \cite[Theorem 2.2~\& p.~34]{Bar10}:
\begin{enumerate}[(i)]
\item $A$ is maximal monotone.
\item $\id+\lambda A$ is surjective for some $\lambda>0$.
\item $\id+\lambda A$ is surjective for all $\lambda>0$.
\end{enumerate}
Consequently, if $A$ is maximal monotone, then $\id+\lambda A$ is bijective for all $\lambda>0$. The Cauchy--Schwarz inequality yields that
\[\|( \id+\lambda A)x-( \id+\lambda A)y\|\geq \|x-y\|\quad \forall x,y\in \dom(A),\]
hence $( \id+\lambda A)^{-1}\colon X\rightarrow X$ is contractive. %

Furthermore, $(\id-\lambda A)(\id+{\lambda} A)^{-1}$ is contractive.
 This follows with $\tilde x\coloneqq (\id+{\lambda} A)^{-1}x$ and $\tilde y\coloneqq (\id+{\lambda} A)^{-1}y$ from
\begin{align*}
\|x-y\|^2-\|(\id -\lambda A)&(\id+{\lambda} A)^{-1}x -(\id-\lambda A)(\id+{\lambda} A)^{-1}y\|^2\\
& =\|\tilde x-\tilde y+{\lambda} A\tilde x-{\lambda} A\tilde y\|^2-\|\tilde x-\tilde y-(\lambda A\tilde x-\lambda A \tilde y)\|^2\\
& =\Re\lambda\langle \tilde x-\tilde y,A\tilde x-A\tilde y\rangle\ge0.
\tag*{$\Box$}
\end{align*}
\item Consequently, if $A$ is linear and maximal dissipative, then $\lambda \id-A$ is surjective for all $\lambda>0$.
\end{enumerate}
\end{remark}

\begin{remark}\label{rem:disssyst}
Let $S = \sbvek{A\& B}{C\& D}\colon \dom(S)\subset\sbvek XU\to\sbvek XU$ be an operator with the properties as specified in Assumptions~\ref{ass:disssyst}.
\begin{enumerate}[(i)]
\item\label{rem:disssyst1} Dissipativity of $\sbvek{\phantom{-}A\& B}{-C\& D}$ directly implies that $A$ is dissipative. Using Remark~\ref{rem:inv} together with $\rho(A)\cap\C_+\neq\emptyset$, $A$ is even maximal dissipative. By the Lumer--Phillips theorem \cite[Chap.~2, Thm.~3.15]{EnNa00}, we obtain that $A$ generates a~strongly continuous semigroup. Consequently, $S$ is a~system node.
\item\label{rem:disssyst2} It follows from \cite[Lem.~4.3]{St02a} that $\sbvek{\phantom{-}A\& B}{-C\& D}$ is maximal dissipative.
\item\label{rem:disssyst3} The transfer function $\widehat{\mathfrak D}$ of $S$ is defined on $\C_+$. Moreover, $\widehat{\mathfrak D}(s)$ is monotone (and thus maximal monotone as it is a bounded operator) for all $s\in\C_+$ \cite[Thm.~4.2]{St02a}. Further, the system
\eqref{eq:ODEnode} is  well-posed if, and only if, $\setdef{\|\widehat{\mathfrak D}(\sigma+\imath\omega)\|}{\omega\in\R}$ is bounded for some (and hence any) $\sigma>0$ \cite[Thm.~5.1]{St02a}.
\item\label{rem:disssyst4} The generalized (and thus also the classical) trajectories of \eqref{eq:ODEnode} fulfill the \emph{dissipation inequality}
\begin{equation}\|x(t)\|_X^2\leq \|x(0)\|_X^2+2\int_0^t\Re\langle u(\tau),y(\tau)\rangle_U\,{\rm d}\tau\qquad\forall\,t\in[0,T],\label{eq:dissineq}\end{equation}
see \cite[Thm.~4.2]{St02a}.
\end{enumerate}
\end{remark}

\begin{lemma}\label{lem:dissflowinv}
Assume that $S = \sebvek{\phantom{[}A\& B\phantom{]_2}}{{[C\& D]_1}}{{[C\& D]_2}}\colon \dom(S)\subset \sebvek X{\mathfrak U}U\to\sebvek X{\mathfrak U}U$ has the properties as specified in Assumptions~\ref{ass:disssyst} and is partially flow-invertible. Then the partial flow inverse $S^\curvearrowleft = \sebvek{[A\& B]^\curvearrowleft}{{[C\& D]_1^\curvearrowleft}}{{[C\& D]_2^\curvearrowleft}}$ fulfills that $\sebvek{\phantom{-}[A\& B]^\curvearrowleft}{{-[C\&D]_1^\curvearrowleft}}{{-[C\& D]_2^\curvearrowleft}}$ is dissipative.
\end{lemma}
\begin{proof}
By Remark~\ref{rem:disssyst}~\eqref{rem:disssyst4}, the generalized trajectories of \eqref{eq:ODEnode} fulfill the dissipation inequality \eqref{eq:dissineq}.
By using Proposition~\ref{prop:trajectories} and the trivial fact that $\Re\langle u(\tau),y(\tau)\rangle=\Re\langle y(\tau),u(\tau)\rangle$ for all $\tau\in[0,T]$, we see that the generalized trajectories for the system associated to the node $S^\curvearrowleft$ again fulfill the dissipation inequality.
Then \cite[Thm.~4.2]{St02a} yields that $\sebvek{\phantom{-}[A\& B]^\curvearrowleft}{{-[C\&D]_1^\curvearrowleft}}{-[C\& D]_2^\curvearrowleft}$ is dissipative.
\end{proof}
\begin{remark}\label{rem:flowinvdiss}
A~consequence of Lemma~\ref{lem:dissflowinv} is that partial flow inverses of system nodes fulfilling Assumptions~\ref{ass:disssyst} again fulfill Assumptions~\ref{ass:disssyst}. In particular, $\sebvek{\phantom{-}[A\& B]^\curvearrowleft}{{-[C\&D]_1^\curvearrowleft}}{{-[C\& D]_2^\curvearrowleft}}$ is maximal dissipative.
\end{remark}

\begin{lemma}\label{lem:strdissflowinv2}
Assume that $S =\left[ \begin{smallmatrix}
{\phantom{[}A\& B\phantom{]_1}}\\ {[C\& D]_1}\\ {[C\& D]_2}
\end{smallmatrix}\right]\colon \dom(S)\subset\sebvek X{\mathfrak U}U\to\sebvek X{\mathfrak Y}Y$ has the properties as specified in Assumptions~\ref{ass:disssyst}, and let
$\gamma\geq 0$, $\delta>0$. Then the system node
\[S_{\gamma,\delta} = \left[\begin{smallmatrix}
    {(A-\gamma \id)\& B}\\ 0\phantom{\&}\id\phantom{\&}0  \\{[C\& (D+\delta \id)]_2}
\end{smallmatrix}\right]\]
is partially flow-invertible. Moreover, for the partial flow-inverse $S_{\gamma,\delta}^{\curvearrowleft}$ of $S_{\gamma,\delta}$, the subsystem $\widetilde{\Sigma}_{22}^\curvearrowleft$ (i.e., the restriction of the system node to $(Y,X,U)$) is well-posed.
Here, we denote the system given by the node $S_{\gamma,\delta}$ by $\widetilde{\Sigma}=\sbmat{\widetilde{\Sigma}_{11}}{\widetilde{\Sigma}_{12}}{\widetilde{\Sigma}_{21}}{\widetilde{\Sigma}_{22}}$ and the one for $S_{\gamma,\delta}^\curvearrowleft$ by $\widetilde{\Sigma}^\curvearrowleft$ respectively.
\end{lemma}
\begin{proof}
Since $\left[\begin{smallmatrix}
    A\& B\\0\ 0\ 0 \\ {-[C\& D]_2}
\end{smallmatrix}\right]$ is maximal dissipative aswell by Remark~\ref{rem:disssyst}~\eqref{rem:disssyst2}, we obtain that
\[\left[\begin{smallmatrix}
    {(\delta-\gamma)\id} &0&0\\0 & \id & 0 \\ 0&0&0
\end{smallmatrix}\right]- \left[\begin{smallmatrix}
    {(A-\gamma \id)\& B}\\ 0  \\{[C\& (D+\delta \id)]_2}
\end{smallmatrix}\right] =-\left[\begin{smallmatrix}
    {(A-\delta \id)\& B}\\ 0\ -\id \ 0  \\{[C\& (D+\delta \id)]_2}
\end{smallmatrix}\right] =\left[ \begin{smallmatrix*}[r]\id &0 &0\\ 0&\tfrac{1}{\delta}\id&0\\0&0&-\id\end{smallmatrix*}\right] \left(\delta \id-\left[ \begin{smallmatrix}{\phantom{-}A\& B\phantom{]_2}}\\0\ 0\ 0 \\{-[C\& D]_2}\end{smallmatrix}\right]\right)\]
has a~bounded inverse, which we partition as $\left[ \begin{smallmatrix} {M_{11}}&{M_{12}}&{M_{13}}\\{M_{21}}&{M_{22}}&{M_{23}}\\{M_{31}}&{M_{32}}&{M_{33}}\end{smallmatrix}\right]$. Moreover, by
\[\Re\left\langle\sebvek x{\mathfrak{u}}u,\left(\delta \id- \left[ \begin{smallmatrix}{\phantom{-}A\& B\phantom{]_2}}\\0\ 0\ 0 \\{-[C\& D]_2}\end{smallmatrix}\right]  \right)\sebvek x{\mathfrak{u}}u\right\rangle \ge\delta\big(\|x\|_X^2+ \| \mathfrak u\|_{\mathfrak U}^2 +\|u\|_U^2\big)\quad\forall\,\sebvek x{\mathfrak{u}}u\in\dom(S),\]
we obtain from the construction of $M_{ij}$, $i,j=1,2,3$ that
\[\Re\Big\langle\sebvek x{\mathfrak{u}}u,\left[ \begin{smallmatrix} {M_{11}}&{M_{12}}&{M_{13}}\\{M_{21}}&{M_{22}}&{M_{23}}\\{M_{31}}&{M_{32}}&{M_{33}}\end{smallmatrix}\right]\sebvek x{\mathrm{\delta} \mathfrak u}{-u}\Big\rangle{ >}0\quad\forall\sebvek x{\mathfrak{u}}u\in \sebvek X {\mathfrak U}  U\setminus\Big\{\sebvek000\Big\}.\]
In particular,
\[\Re\langle x,M_{11}x\rangle>0\quad\forall x\in X\setminus\{0\},\]
hence $M_{11}$ is injective and has dense range. This together with the boundedness of $M_{11}$ implies that the inverse of $-M_{11}$ is again
maximal dissipative, and the Lumer--Phillips theorem \cite[Chap.~2, Thm.~3.15]{EnNa00} yields that $-M_{11}^{-1}$ generates a~strongly continuous semigroup on
$X$.
Now we can conclude from Proposition~\ref{prop:sysnodepartflowinv} that $S_{\gamma,\delta}$ possesses a~partial flow inverse $S_{\gamma,\delta}^\curvearrowleft $.\\
It remains to prove that $S_{\gamma,\delta}^\curvearrowleft$ defines a~well-posed subsystem $\widetilde{\Sigma}_{22}^\curvearrowleft$. Let $\widehat{\mathfrak D}= \sbmat {\widehat{\mathfrak D}_{11}}{\widehat{\mathfrak D}_{12}}{\widehat{\mathfrak D}_{21}}{\widehat{\mathfrak D}_{22}}$ be the transfer function of $S$. Then $s\mapsto \delta \id+\widehat{\mathfrak D}_{22}(\gamma+s)$ is the transfer function of $\widetilde{\Sigma}_{22}$. On the other hand, by Remark~\ref{rem:disssyst}~\eqref{rem:disssyst3}, $\widehat{\mathfrak D}_{22}(\gamma+s)$ is monotone for all $s\in\C_+$, which gives rise to
\[\|(\delta \id+\widehat{\mathfrak D}_{22}(\gamma+s))^{-1}\|\leq\tfrac1\delta\quad\forall \,s\in\C_+.\]
On the other hand, since $(\delta \id+\widehat{\mathfrak D}_{22}(\gamma+s))^{-1}$ is the transfer function of the subsystem $\widetilde{\Sigma}_{22}^\curvearrowleft$ by Definition~\ref{def:partflow}, we can conclude from Remark~\ref{rem:disssyst}\,\eqref{rem:disssyst3} that this subsystem is well-posed.
\end{proof}

Let $S = \left[ \begin{smallmatrix}
    \phantom{[} A\& B\phantom{]_1} \\{[C\& D]_1} \\ {[C \& D]_2} 
\end{smallmatrix}\right]\colon \dom(S)\subset \sebvek X{\mathfrak U}U\to \sebvek X{\mathfrak U}U$ be as in Assumptions~\ref{ass:disssyst}, and let $T>0$, $\mathfrak u\colon [0,T]\to \mathfrak{U} $ and $x_0\in X$. We recall the definition of the operator
\begin{subequations}\label{eq:Mop}
\begin{equation}
M\colon \dom(M)\subset H \coloneqq \LL^2([0,T];\sbvek XU)\to H
\end{equation}
with
\begin{equation}
\dom(M)=\setdef{\sbvek xu\in H}{\parbox[c]{6.2cm}{$\displaystyle{\sbvek{\dot{x}}0-\left[ \begin{smallmatrix}
\phantom{-[} A\& B \phantom{]_1} \\ {-[C\& D]_2}
\end{smallmatrix}\right] \sebvek x{\mathfrak u}u}\in H$ and $x(0)=x_0$}},
\end{equation}
\begin{equation}
M\bvek xu=\bvek{\dot{x}-A\&B\sebvek x{\mathfrak u}u}{{[C\&D]_2}\sebvek x{\mathfrak u}u}.
\end{equation}
\end{subequations}
\begin{remark}\label{rem:sol}
By $\sbvek{\dot{x}}0-\left[ \begin{smallmatrix}
   \phantom{[} A\&B\phantom{]_2} \\ {[C\&D]_2} \end{smallmatrix}\right] \sebvek x{\mathfrak u}u\in \LL^2([0,T];\sbvek X U)$, we mean that there exist $w\in \LL^2([0,T];X)$, $z\in \LL^2([0,T];U)$ such that $\sebvek x{\mathfrak u}u$ fulfills
$\dot{x}=A\&B \sebvek x{\mathfrak u}u+w$ and $z=[C\&D]_2\sebvek x{\mathfrak u}u$ in the sense of generalized trajectories in Definition \ref{def:traj}. These functions indeed fulfill
\[\sbvek wz=M\sbvek xu.\]
As the action of $M$ is defined via generalized trajectories, Remark~\ref{rem:sols}\,\eqref{rem:sols2} yields that $x\in \Ce([0,T];X_{-1})$, hence the initial condition $x(0)=x_0$ is well-defined.
Hence, generalized trajectories of the closed system \eqref{eq:singlesystem1}-\eqref{eq:coupling} are equivalent to solutions to the equation \eqref{eq:algsplitting}.
\end{remark}
Our main result of this section is presented in the following.
\begin{theorem}\label{thm:Mmon}
Let $S\colon \dom(S)\subset 
\sebvek X{\mathfrak U}U\to \sebvek X{\mathfrak U}U$ 
be as in Assumptions~\ref{ass:disssyst}, and let $T>0$, $\mathfrak u\in \LL^2([0,T];\mathfrak {U})$ and $x_0\in X$. Then the operator
$M$ as in \eqref{eq:Mop} is closed and maximal monotone.
\end{theorem}
\begin{proof}

\emph{Step~1:} We show that $M$ is monotone. Let $\sbvek {x_1}{u_1},\sbvek {x_2}{u_2}\in\dom(M)$. Denote
\[\sbvek {w_i}{z_i}\coloneqq M\sbvek {x_i}{u_i},\quad i=1,2.\]
Then $\sbvek {w}{z}\coloneqq \sbvek {w_1-w_2}{z_1-z_2}$, $\sbvek {x}{u}\coloneqq \sbvek {x_1-x_2}{u_1-u_2}$ fulfill $x(0)=0$, 
and
\[\sbvek{\dot{x}}{z}=\sbvek{\phantom{[}A\&B\phantom{]_2}}{{[C\&D]_2}}\sebvek{{x}}{\mathrm 0}{u}+\sbvek{w}{0},\]
which gives
\[
\left[\begin{smallmatrix}\dot{x}\\z\\x\end{smallmatrix}\right]=\underbrace{\left[\begin{smallmatrix}\phantom{[}A \& B\phantom{]_2}& \id\\{[C\& D]_2}&0\\\id\phantom{\&} 0\phantom{\&} 0&0\end{smallmatrix}\right]}_{\eqqcolon S_{\ext}}\left[\begin{smallmatrix}{x}\\ 0\\u\\w\end{smallmatrix}\right],\]
in the sense of generalized solutions.
Since $S$  fulfills Assumptions~\ref{ass:disssyst}, it is straightforward to see that so does $S_{\ext}$, too.
Then the dissipation inequality (see Remark~\ref{rem:disssyst}\,\eqref{rem:disssyst4}) yields
\begin{align*}
  0\leq &\,\tfrac12\|x(T)\|^2_X=\tfrac12\|x(T)\|^2_X-\tfrac12\|x(0)\|^2_X \\
 \leq &\,\int_0^T\Re\left\langle \sbvek{u(\tau)}{w(\tau)},\sbvek{z(\tau)}{x(\tau)}\right\rangle_{\smbvek  UX}{\rm d}\tau\\
 =&\,\int_0^T\Re\left\langle \sbvek{u(\tau)}{x(\tau)},\sbvek{z(\tau)}{w(\tau)}\right\rangle_{\smbvek  UX}{\rm d}\tau\\
 =&\,\int_0^T\Re\left\langle \sbvek{x(\tau)}{u(\tau)},\sbvek{w(\tau)}{z(\tau)}\right\rangle_{\smbvek  XU}{\rm d}\tau\\
 =&\, \int_0^T\Re\left\langle \sbvek{x_1(\tau)-x_2(\tau)}{u_1(\tau)-u_2(\tau)},\sbvek{w_1(\tau)-w_2(\tau)}{z_1(\tau)-z_2(\tau)}\right\rangle_{\smbvek X U}{\rm d}\tau\\
 =&\, \Re\left\langle \sbvek{x_1}{u_1}-\sbvek{x_2}{u_2},M\sbvek{x_1}{u_1}-M\sbvek{x_2}{u_2}\right\rangle_{\LL^2([0,T];\smbvek X U)}.
\end{align*}
\emph{Step~2:} By Remark \ref{rem:inv} (ii) we only need to show that  $\lambda \id+M$ is surjective for any given $\lambda>0$. So take $\lambda>0$, $w\in \LL^2([0,T];X)$, $z\in \LL^2([0,T];U)$.
Lemma~\ref{lem:strdissflowinv2} implies that the system node
\[S_{\lambda,\lambda} = \sebvek{(A-\lambda \id)\& B}{0\phantom{\&}\id \phantom{\&}0}{{[C\& (D+\lambda \id)]_2}}\]
is partially flow-invertible. Denote the partial flow inverse by 
\[S_{\lambda,\lambda}^{\curvearrowleft}=\sebvek{[\widetilde{A}\&\widetilde{B}]^\curvearrowleft}{0\ \id \ 0\phantom{\&}}{{[\widetilde{C}\&\widetilde{D}]_2^\curvearrowleft}}.\]
Lemma~\ref{lem:strdissflowinv2} implies that this further defines a~well-posed subsystem

Since the closed system has a solution, for fixed $\mathfrak u$ there is a (generalized) trajectory $(x_{\mathfrak u},u_{\mathfrak u}, y_{\mathfrak u})$. Due to the well-posedness of the subsystem we can choose $\widehat{u}=u-u_{\mathfrak u}$ as the new input of the subsystem and again obtain a well-defined trajectory. Doing this for arbitrary inputs yields that for fixed $\mathfrak u$ and for all internal inputs $u$ there is trajectory.
Then Lemma~\ref{lem:systext} yields that
\[S_{\lambda,\lambda,\ext}^\curvearrowleft=\left[\begin{smallmatrix}[\widetilde{A}\&\widetilde{B}]^\curvearrowleft & \id\\ 0\phantom{\&}\id\phantom{\&} 0&0\\ {[\widetilde{C}\& \widetilde{D}]_2^\curvearrowleft} &0 \\\id\phantom{\&}0\phantom{\&}0&0\end{smallmatrix}\right]\]
has the same well-posedness property. Hence, there exist $x\in \Ce([0,T];X)$ with $x(0)=x_0$ and $u\in \LL^2([0,T];U)$ with 
\begin{equation}\left[\begin{smallmatrix}\dot{x}\\\mathfrak u\\u\\x\end{smallmatrix}\right]=S_{\lambda,\lambda,\ext}^\curvearrowleft\left[\begin{smallmatrix}x\\\mathfrak u \\z\\w\end{smallmatrix}\right],\label{eq:extinvsys}\end{equation}
and thus,
\[\left[\begin{smallmatrix}\dot{x}\\ \mathfrak u\\u\end{smallmatrix}\right]=S_{\lambda,\lambda}^\curvearrowleft\left[\begin{smallmatrix}x\\ \mathfrak u\\z\end{smallmatrix}\right]+\left[\begin{smallmatrix}w\\0\\0\end{smallmatrix}\right].\]
The definition of partial flow inverse yields
\[\left[\begin{smallmatrix}\dot{x}\\ \mathfrak u\\u\end{smallmatrix}\right]=\sebvek{(A-\lambda \id)\&B}{0\quad\id\quad 0}{0 \quad 0 \quad\id}\sebvek{\id\quad 0\quad 0}{0\quad\id\quad 0}{{[C\&(D+\lambda \id)]_2}}^{-1}\left[\begin{smallmatrix}x\\\mathfrak u\\z\end{smallmatrix}\right]+\left[\begin{smallmatrix}w\\0\\0\end{smallmatrix}\right]\]
This yields together with $z=[C\&(D+\lambda \id)]_2\sebvek x{\mathfrak u}u$ 
\[\left[\begin{smallmatrix}\dot{x}\\\mathfrak u\\z\end{smallmatrix}\right]=S_{\lambda,\lambda}\left[\begin{smallmatrix}x\\\mathfrak u\\u\end{smallmatrix}\right]+\left[\begin{smallmatrix}w\\0\\0\end{smallmatrix}\right]\]
with $x(0)=x_0$. Since the second line of this equation is redundant, this is equivalent to 
\[\sbvek {\dot x}{z} = \sbvek{(A-\lambda \id)\& B}{{[C\&(D+\lambda \id)]_2}}\sebvek x{\mathfrak u}u + \sbvek w0. \]
By definition of $M$, this means that $(\lambda \id+M)\sbvek xu=\sbvek wz$ and with Remark~\ref{rem:inv} we obtain the maximal monotonicity of the operator.
 \end{proof}

\section{Proof of the main theorems}\label{sec:proof}

This section is devoted to the proof of Theorem~\ref{thm:convergence} and \ref{thm:outputconvergence}.

\begin{proof}[Proof of Theorem \ref{thm:convergence}]
Let again $\sbvek{x_k}{u_k} \in \dom(M)$ and denote $\sbvek{w_k}{z_k} \coloneqq (\id + \lambda M)\sbvek{x_k}{u_k}$.
Then, following the proof of \cite[Thm.~23]{BaGuJaRe22} we have
\begin{align}
    \kvek wz{k+1} &= (\id + \lambda M)(\id + \lambda M)^{-1}(\id - \lambda N)(\id + \lambda N)^{-1} (\id-\lambda M)\kvek xuk\notag\\
    &= (\id - \lambda N)(\id + \lambda N)^{-1} (\id-\lambda M)(\id + \lambda M)^{-1} \kvek wzk\notag\\
    &= (\id - \lambda N)(\id + \lambda N)^{-1} (2\id-(\id +\lambda M))(\id + \lambda M)^{-1} \kvek wzk\notag\\
    &= (\id - \lambda N)(\id + \lambda N)^{-1} (2\kvek xuk - \kvek wzk)
\end{align}
and analogously 
\[ \sbvek wz = (\id - \lambda N)(\id + \lambda N)^{-1} (2\sbvek xu - \sbvek wz). \]

Let $\omega > 0$  and $\LL^2_\omega([0,T];X)\coloneqq
        \{f\colon [0,T]\rightarrow X\mid \e^{-\cdot \omega} f(\cdot) \in  \LL^2([0,T];X)\}$
equipped with the norm $\| \cdot \|_{2,\omega,X}$ defined as
\[
   \|f\|_{2,\omega,X}^2= \int_0^T \e^{-2t \omega} \| f(t)\|^2_X\ds t.
\]
Analogously, define $\LL^2_{\omega}([0,T];  \sbvek X U)$ with norm $\| \cdot \|_{2,\omega, \smbvek X U}$.
Clearly, as  $T\in (0,\infty)$, we have $\LL^2_\omega([0,T];X)=\LL^2([0,T];X)$ with equivalent norms. 

Let $\sbvek {x_1}{u_1},\sbvek {x_2}{u_2}\in\dom(M)$. The dissipativity inequality
\begin{align*}
\ddts \tfrac12\|x_1(t)-x_2(t)\|^2_X
 \le &\, \Re\left\langle \sbvek{x_1(t)}{u_1(t)}-\sbvek{x_2(t)}{u_2(t)},M\sbvek{x_1}{u_1}(t)-M\sbvek{x_2}{u_2}(t)\right\rangle_{\smbvek XU}
\end{align*}
implies
\begin{align*}
\MoveEqLeft[8] \Re\left\langle \sbvek{x_1}{u_1}-\sbvek{x_2}{u_2},M\sbvek{x_1}{u_1}-M\sbvek{x_2}{u_2}\right\rangle_{2,\omega,\smbvek XU}\\
&\ge  \e^{-2\omega T}\tfrac12\|x_1(T)-x_2(T)\|^2_X +\omega \|x_1-x_2\|^2_{2,\omega,X}.
\end{align*}

Denote 
\[ \Delta \kvek xuk \coloneqq \sbvek xu - \kvek xuk, \qquad \Delta \kvek wzk \coloneqq \sbvek wz - \kvek wzk, \]
then, using the monotonicity of $M$ and $N$ we obtain for $\omega >0$
\begin{align}\label{eq:monzseq}
    \MoveEqLeft[1] \| \Delta \kvek wz{k+1}\|_{2,\omega,\smbvek XU}^2 - \|\Delta \kvek wzk\|_{2,\omega,\smbvek XU}^2 \notag\\
    &= \| (\id - \lambda N)(\id + \lambda N)^{-1}\left[(2\sbvek xu-\sbvek wz) -  (2\kvek xuk -\kvek wzk) \right]\|_{2,\omega,\smbvek XU}^2 \notag\\
    &\quad - \|\Delta \kvek wzk\|_{2,\omega,\smbvek XU}^2 \notag\\
    &\le  \| 2\Delta \kvek xuk - \Delta \kvek wzk\|_{2,\omega,\smbvek XU}^2 - \|\Delta \kvek wzk\|_{2,\omega,\smbvek XU}^2 \notag \\
    &= 4 \| \Delta \kvek xuk \|_{2,\omega,\smbvek XU}^2 - 4 \Re\left\langle \Delta \kvek xuk , \Delta \kvek wzk \right\rangle_{2,\omega,\smbvek XU} \notag\\
    &= 4 \| \sbvek xu- \kvek xuk \|_{2,\omega,\smbvek XU}^2 - 4\Re \langle \sbvek xu-\kvek xuk ,(\id + \lambda M)\sbvek xu -(\id + \lambda M)\kvek xuk \rangle_{2,\omega,\smbvek XU} \notag\\
    &= -4 \lambda \Re\langle \sbvek xu-\kvek xuk ,M\sbvek xu-M\kvek xuk  \rangle_{2,\omega,\smbvek XU} \notag \\
    &\le -2\lambda \e^{-2\omega T}\|x(T)-x_k(T)\|^2_X -4\lambda \omega \|x-x_k\|^2_{2,\omega,X} \le 0.
\end{align}
Hence, $ \| \kvek wzk - \sbvek wz\|_{2, \omega,\smbvek XU}$ is monotonically decreasing (therefore convergent) and Remark \ref{rem:inv} yields the required inequality for $\LL^2_{\omega}([0,T];X)$.
Since the monotonicity of $M$ also holds in $\LL^2([0,T];X)$, the same statement is true for $\omega = 0$.

Now rephrasing the last inequality gives 
\[ \| x-x_k \|_{2,\omega,X}^2 \le \tfrac{1}{4\lambda \omega} \left( \|\Delta \kvek wzk\|_{2,\omega,\smbvek XU}^2 - \| \Delta \kvek wz{k+1}\|_{2,\omega,\smbvek X U}^2\right), \]
which leads to \textit{b)} since the norms of $\LL^2_{\omega}([0,T];X)$ and $\LL^2([0,T];X)$ are equivalent.\\
Lastly, we repeat the calculation \eqref{eq:monzseq} but with $t \in (0,T]$ instead of $T$ in the last step and obtain
\[ \| x(t) - x_k(t) \|_X^2 \le \frac{\e^{2\omega T}}{2\lambda} \left( \|\Delta \kvek wzk\|_{2,\omega,\smbvek XU}^2 - \| \Delta \kvek wz{k+1}\|_{2,\omega,\smbvek XU}^2\right), \]
which completes the proof.
\end{proof}

\begin{proof}[Proof of Theorem \ref{thm:outputconvergence}] 
Let $\kvek xuk \in \dom(M)$, $(x,\mathfrak u, \mathfrak y)$ be a generalized trajectory and $\mathfrak y_k \coloneqq [C\&D]_1 \sebvek {x_k}{\mathfrak u}{u_k}$.
Using the (PSOP) property, the dissipation inequality refines to
\begin{align*}
\MoveEqLeft
\Re \left\langle \sbvek xu - \kvek xuk, M\sbvek xu - M\kvek xuk  \right\rangle_{2,\smbvek XU}\\
&\ge \tfrac{1}{2}\|x(T)-x_k(T)\|_X^2 - \Re \left\langle \sbvek xu - \kvek xuk, \sbvek {{A\&B}}{{-[C\&D]_2}}\left( \sebvek x{\mathfrak u}u - \sebvek {x_k}{\mathfrak u}{u_k}\right) \right\rangle_{2,\smbvek XU}  \\
&= \tfrac{1}{2}\|x(T)-x_k(T)\|_X^2 - \Re \left\langle \sebvek x{\mathfrak u}u - \sebvek {x_k}{\mathfrak u}{u_k}, \sebvek {{A\&B}}{{-[C\&D]_1}}{{-[C\&D]_2}}\left( \sebvek x{\mathfrak u}u - \sebvek {x_k}{\mathfrak u}{u_k}\right) \right\rangle_{2,\sembvek X{\mathfrak U}U}  \\
&\ge \tfrac{1}{2} \|x(T)-x_k(T)\|_X^2 + \eps \|\mathfrak y - \mathfrak y_k\|_{2,\mathfrak U}^2.
\end{align*}
Hence, \eqref{eq:monzseq} can be reformulated to
\begin{align*}
    \|\mathfrak y - \mathfrak y_k\|_{2,\mathfrak U}^2 \le \tfrac{1}{4\eps} \left(  \|\Delta \kvek wzk\|_{2,\smbvek XU}^2 - \| \Delta \kvek wz{k+1}\|_{2,\smbvek X U}^2  \right).
\end{align*}
Again, due to the convergence of $\|\Delta \kvek wzk\|_{2,\smbvek XU}^2$ (from Theorem \ref{thm:convergence}) we obtain the desired convergence of the external outputs.
\end{proof}

\section{Examples}
\subsection{A coupled wave-heat system}\label{sec:ex}

We consider the following one-dimensional coupled wave-heat system with Coleman--Gurtin thermal law \cite{DePaSe22}
\begin{align*}
v_{tt}(\zeta,t)&=v_{\zeta \zeta}(\zeta,t),&\zeta\in(-1,0), t>0,\\
w_t(\zeta,t)&=w_{\zeta \zeta}(\zeta,t)+\int_0^\infty g(s)w_{\zeta \zeta}(\zeta,t-s)\,{\rm d}s,&\zeta\in(0,1), t>0,\\
v_t(0,t)&=w(0,t),& t>0,\\
v_\zeta(0,t)&=w_\zeta(0,t)+\int_0^\infty g(s)w_{\zeta}(0,t-s)\,{\rm d}s,& t>0,\\
v(-1,t)&=0,& t>0,\\
w(1,t)&=0,& t>0,
\end{align*}
equipped with initial conditions
\begin{align*}
v(\zeta,0)&= v_0(\zeta),& v_t(\zeta,0)&= \psi(\zeta), & \zeta \in (-1,0),\\
w(\zeta,0)&= w_0(\zeta), & w(\zeta,-s)&= \varphi_0(\zeta,s), & \zeta \in (0,1), s>0,
\end{align*}
for suitable functions $v_0$, $w_0$, $\psi_0$ and $\varphi_0$.
Here the convolution kernel $g\colon [0,\infty)\rightarrow [0,\infty)$ is convex and integrable with unit total mass and of the form
\begin{align*}
g(s)&=\int_s^\infty \mu \,{\rm d}r,&s\ge 0,
\end{align*}
and $\mu\colon (0,\infty)\rightarrow [0,\infty)$ is non-increasing, absolutely continuous and integrable. Further, details of the model and the corresponding analysis can be found in \cite{DePaSe22}.

The coupled wave-heat system can be decomposed into two impedance passive system nodes which are coupled in a power conserving manner. More precisely, the wave part of the system is given by
\begin{align*}
v_{tt}(\zeta,t)&=v_{\zeta \zeta}(\zeta,t),&\zeta\in(-1,0), t>0,\\
v(-1,t)&=0   ,\quad v_t(0,t)=u_1(t),\quad y_1(t)=v_\zeta(0,t), &t>0,\\
v(\zeta,0)&= v_0(\zeta), \quad v_t(\zeta,0)= \psi(\zeta), & \zeta \in (-1,0).
\end{align*}
In \cite{DePaSe22} it has been shown that the wave part 
is an impedance passive system node $S_1=\sbvek{{A_1}\&{B_1}}{{C_1}\&{D_1}}$ on $\sebvek{\mathbb C}{\sbvek{\Hlr{\ell}}{\LL^2(-1,0)}}{\mathbb C}$, where 
\[\Hlr \ell\coloneqq \{v\in \Hlr{} \mid v(-1)=0\}.\] 
The state of the system is given by $\sbvek{v(\cdot,t)}{v_t(\cdot,t)}$. As $S_1$ is impedance passive, the operator $\sbvek{\phantom{-}{A_1}\&{B_1}}{{-C_1}\&{D_1}}$ is dissipative, see \cite[Theorem 4.2]{St02a}.

 The heat part with Coleman--Gurtin thermal law is described by
\begin{align*}
 w_t(\zeta,t)&=w_{\zeta \zeta}(\zeta,t)+\int_0^\infty \!\!g(s)w_{\zeta \zeta}(\zeta,t-s)\,{\rm d}s,&\zeta\in(0,1), t>0,\\  
 w(1,t)&=0,&t>0,\\
 \quad u_2(t)&=
 -w_\zeta(0,t)-\int_0^\infty\!\! g(s)w_{\zeta}(0,t-s)\,{\rm d}s,&t>0,\\
  y_2(t)&=w(0,t),&t>0,\\
 w(\zeta,0)&= w_0(\zeta), \quad  w(\zeta,-s)= \varphi_0(\zeta,s), & \zeta \in (0,1), s>0.
\end{align*}
The heat part is an impedance passive system system node $S_2=\sbvek{{A_2}\&{B_2}}{{C_2}\&{D_2}}$ on $(\mathbb C,\LL^2(0,1)\times \mathcal{M}, \mathbb C)$, see \cite{DePaSe22}. Here $\mathcal{M}\coloneqq \LL_\mu^2((0,\infty); \Hlr r)$, where $\LL_\mu^2$ is the space of all square-integrable functions with respect to the measure $\mu(s){\rm d}s$ and $\Hlr r\coloneqq \{w\in \Hlr{}\mid w(1)=0\}$. The state of the system is given by $\sbvek{w(\cdot,t)}{s\mapsto \int_0^s w(\cdot,t-\sigma)\, {\rm d}\sigma}$. 

Thus Assumption \ref{ass:disssyst} is satisfied and as the two systems are coupled via $u_1(t)=y_2(t)$ and $u_2(t)=-y_1(t)$ the corresponding coupling operator $N$ is skew-adjoint and thus dissipative.  Therefore our dynamic iteration scheme is applicable and it is possible to solve the heat and wave part independently and in parallel via suitable methods.

\begin{remark}
    One could also include external inputs to the system by setting the value of $v$ or $w$ equal to a given $\mathfrak u$ instead of 0.
\end{remark}

\subsection{Wave equation on an L-shaped/decomposable domain}
We consider a wave equation as in \cite{KuZw15} and use our technique to decompose the domain with a coupling on the connecting boundary.
The system is given in the following form:
\begin{align*}
    \rho(\xi) z_{tt}(\xi,t) &= \Div(T(\xi) \grad z(\xi,t)) - \left( d z_t \right) (\xi,t), \quad  \xi \in \Omega, t \ge 0, \\
    0 &= z_t (\xi ,t ) \quad \text{on } \Gamma_0 \times [0,\infty),\\
    {\mathfrak u}(\xi,t) &= \nu \cdot (T(\xi)\grad z(\xi,t)) \quad \text{on } \Gamma_1 \times [0,\infty),\\
    {\mathfrak y}(\xi,t) &= z_t(\xi,t) \quad \text{on } \Gamma_1 \times [0,\infty), \\
    z(\xi,0)&= z_0(\xi), \quad z_t(\xi,0) = \omega (\xi) \quad \text{on } \Omega,
\end{align*}
where $\nu\colon \partial\Omega\to\R^2$ is the unit outward normal vector of an L-shaped domain $\Omega\subset\R^2$. More precisely, we consider
$\Omega,\Gamma_0,\Gamma_1 \subseteq \R^2$, with
\begin{align*}
\Omega&={\mathrm{int}}(\overline{\Omega_1}\cup\overline{\Omega_2}),\\
\Omega_1&=(0,1)\times(0,2),&
\Omega_2&=(1,2)\times(0,1),\\
\Gamma_1&=(0,1)\times\{2\}\subset\partial\Omega,&
\Gamma_0&=\partial\Omega\setminus\widebar \Gamma_1.
\end{align*}
As usual, $z(\xi,t)$ denotes the displacement of the wave at point $\xi \in \Omega$ and time $t\ge 0$, $\mathfrak u$ is the input given by a force on the boundary part $\Gamma_1$ and $\mathfrak{y}$ is the output measured as the velocity at $\Gamma_1$. The physical parameters are included via the Young's modulus $T(\cdot)$ and the mass density $\rho(\cdot)$ (not to be confused with the resolvent set), which are both assumed to be measurable, positive, and they have a bounded inverses.
$d$ can be interpreted as an internal damping which is assumed to be a bounded nonnegative and measurable function on $\Omega$  (often given by a multiplication operator). 

We split the problem into two wave equations, each on the rectangles $\Omega_1$ and $\Omega_2$, which interact at the boundary interface
\[\Gamma_{\rm int}\coloneqq \{1\}\times(0,1).\]
First note that, for the sets
\begin{align*}
\Gamma_{10}&=\big(\{0\}\times[0,2)\big)\cup\big([0,1)\times\{0\}\big)\cup\big(\{1\}\times(1,2)\big),\\
\Gamma_{20}&=\big((1,2]\times\{0\}\big)\cup\big(\{2\}\times[0,1]\big)\cup\big([1,2]\times\{1\}\big).
\end{align*}
we have $\partial\Omega=\partial\Omega_1\cup\partial\Omega_2\setminus\Gamma_{\rm int}$ as well as $\partial\Omega_2=\widebar{\Gamma}_{20}\cup\widebar{\Gamma}_{\rm int}$ and $\partial\Omega_1=\widebar{\Gamma}_{10}\cup\widebar{\Gamma}_{\rm int}\cup \widebar{\Gamma}_1$ (see Figure \ref{fig:omega}). Further, the set
$\widebar \Gamma_{10}\cap\widebar \Gamma_{20}=\widebar \Gamma_{10}\cap\widebar \Gamma_{\rm int}=\widebar \Gamma_{20}\cap\widebar \Gamma_{\rm int} = \{1\}\times \{0,1\}$ is of line measure zero.\\

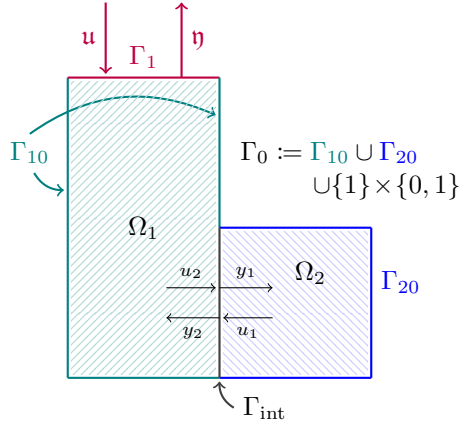
\begin{figure}
    \centering
    
\begin{tikzpicture}
  \draw[teal, thick] (0,0)--+(0,4);
  \node[teal] (Gamma10) at (-0.5,3) {$\Gamma_{10}$};
  \draw[->, teal, thick] (Gamma10) to[bend right=30] (-0.05,2.5);
   \draw[->, teal, thick] (-0.5,3.25) to[bend left=30] (1.95,3.5);
  \draw[purple, thick] (0,4)--+(2,0) node[above, midway]{$\Gamma_{1}$} ;
  \draw[teal, thick] (2,4)--+(0,-2);
  \draw[blue, thick] (2,2)--+(2,0);
  \draw[blue, thick] (4,2)--+(0,-2) node[midway, above right]{$\Gamma_{20}$};
  \draw[blue, thick] (2,0)--+(2,0);
  \draw[teal, thick] (0,0)--+(2,0);
  \draw[darkgray, thick] (2,2)--+(0,-2);

  \fill[pattern color=teal!30, pattern=north east lines] (0.05,0.05) rectangle (1.95,3.95);
  \node at (1,2) {$\Omega_{1}$};
  \fill[pattern color=blue!20, pattern=north west lines] (2.05,0.05) rectangle (3.95,1.95);
  \node at (3.2,1.4) {$\Omega_{2}$};
  \node at (3.45,3) {$\Gamma_{0} \coloneqq {\color{teal}\Gamma_{10}} \cup {\color{blue}\Gamma_{20}}$};
  \node at (4.2,2.55) {$\cup \{1\}\!\times\! \{0,1\}$};
  \node (Gammaint) at (2.6,-0.4) {$\Gamma_{\mathrm{int}}$};
   \draw[->, darkgray, thick] (Gammaint) to[bend left=30] (2,-0.05);

  \draw[->] (1.3,1.2)--+(0.65,0) node[midway, above] {\footnotesize $u_{2}$};
  \draw[->] (2,1.2)--+(0.7,0) node[midway, above] {\footnotesize $y_{1}$};
  \draw[->] (2.7,0.8)--+(-0.65,0) node[midway, below] {\footnotesize $u_{1}$};
  \draw[->] (2,0.8)--+(-0.7,0) node[midway, below] {\footnotesize $y_{2}$};

  \draw[->,purple, thick] (0.5,5) --(0.5,4.05)  node[midway, left]{$\mathfrak u$};
  \draw[<-,purple, thick] (1.5,5) --(1.5,4.)  node[midway, right]{$\mathfrak y$};

\end{tikzpicture}
    \caption{Illustration of the partition of $\Omega$ and its boundaries}
    \label{fig:omega}
\end{figure}
Now, by denoting the unit outward normals of $\Omega_1$ and $\Omega_2$ respectively by $\nu_1$ and $\nu_2$, we consider the two systems
\begin{equation}\label{eq:wave1}
\begin{aligned}
    \rho(\xi) z_{tt}(\xi,t) &= \Div(T(\xi) \grad z(\xi,t)) - \left( d z_t \right) (\xi,t), \quad  \xi \in \Omega_1, t \ge 0, \\
    0 &= z_t (\xi ,t ) \quad \text{on } \Gamma_{10} \times [0,\infty),\\
    {\mathfrak u}(\xi,t) &= \nu_1 \cdot (T(\xi)\grad z(\xi,t)) \quad \text{on } \Gamma_1 \times [0,\infty),\\
    {u}_1(\xi,t) &= \nu_1 \cdot (T(\xi)\grad z(\xi,t)) \quad \text{on } \Gamma_{\rm int} \times [0,\infty),\\
    {\mathfrak y}(\xi,t) &= z_t(\xi,t) \quad \text{on } \Gamma_1 \times [0,\infty), \\
    {y_1}(\xi,t) &= z_t(\xi,t) \quad \text{on } \Gamma_{\rm int} \times [0,\infty), \\
    z(\xi,0)&= z_0(\xi), \quad z_t(\xi,0) = \omega (\xi) \quad \text{on } \Omega_1,
\end{aligned}
\end{equation}
and
\begin{equation}\label{eq:wave2}
\begin{aligned}
    \rho(\xi) z_{tt}(\xi,t) &= \Div(T(\xi) \grad z(\xi,t)) - \left( d z_t \right) (\xi,t), \quad  \xi \in \Omega_2, t \ge 0, \\
    0 &= z_t (\xi ,t ) \quad \text{on } \Gamma_{20} \times [0,\infty),\\
    {u_2}(\xi,t) &= z_t(\xi,t) \quad \text{on } \Gamma_{\rm int} \times [0,\infty), \\
    {y}_2(\xi,t) &= \nu_2 \cdot (T(\xi)\grad z(\xi,t)) \quad \text{on } \Gamma_{\rm int} \times [0,\infty),\\
    z(\xi,0)&= z_0(\xi), \quad z_t(\xi,0) = \omega (\xi) \quad \text{on } \Omega_2,
\end{aligned}
\end{equation}
together with the interface conditions
\begin{equation}\label{eq:waveint}
u_1(\xi,t)=-y_2(\xi,t),\;\;
u_2(\xi,t)=y_1(\xi,t), \quad \text{on } \Gamma_{\rm int} \times [0,\infty).
\end{equation}
\par\noindent\textbf{Trace spaces.} 
In this paragraph $\Omega$ denotes a general bounded Lipschitz domain in $\R^2$. The results here will be applied later to the domains described previously. 
To properly introduce the right formulation and spaces for the above systems, consider the {\em trace operator} $\gamma\colon  \HH^{1}(\Omega)\to \HH^{1/2}(\partial\Omega)$ which maps $x\in \HH^{1}(\Omega)$ to its boundary
trace $x|_{\partial\Omega}$, where $\HH^{1/2}(\partial\Omega)$ denotes the Sobolev space of fractional order $1/2$ \cite{adams2003sobolev}. By the {\em trace
theorem} \cite[Thm.~1.5.1.3]{Gris85}, $\gamma$ is bounded and surjective. Further, $\HH(\Div,\Omega)$ is the space of all square integrable functions
whose weak divergence exists and is square integrable. That is, for $\HH^{1}_0(\Omega)\coloneqq \ker\gamma$,
\[z=\Div x\quad\Longleftrightarrow\quad \forall \varphi\in
\HH^{1}_0(\Omega):\;-\langle\grad\varphi,x\rangle_{\LL^2(\Omega;\C^2)}=\langle\varphi,z\rangle_{\LL^2(\Omega)}.\]
Defining $\HH^{-1/2}(\partial\Omega)\coloneqq \HH^{1/2}(\partial\Omega)^*$ with respect to the pivot space $\LL^2(\partial \Omega)$, the {\em normal trace} of $x\in \HH(\Div,\Omega)$ is well-defined by $w=\gamma_N x\in
\HH^{-1/2}(\partial\Omega)$ with
\[
 \forall z\in \HH^{1}(\Omega):\;\langle w,\gamma z\rangle_{\HH^{-1/2}(\partial\Omega),\HH^{1/2}(\partial\Omega)}=\langle \Div
 x,z\rangle_{\LL^2(\Omega)}+\langle x,\grad z\rangle_{\LL^2(\Omega;\C^2)}.\]
 Green's formula \cite[Chap.~16]{Tart07} yields that, indeed $w(\xi)=\nu(\xi)^\top x(\xi)$ for all $\xi\in\partial\Omega$, if $\Omega$ and $x$ are smooth.
 Further, $\gamma_N\colon \HH(\Div,\Omega)\to \HH^{-1/2}(\partial\Omega)$ is bounded and surjective \cite[Lem.~20.2]{Tart07}.\\
For a~relatively open set $\Gamma\subset\partial\Omega$, we consider
\[\mathrm H^1_{\Gamma}(\Omega) \coloneqq \setdef{f \in \mathrm H^1(\Omega)}{(\gamma f)\vert_{\Gamma}=0 \text{ in }\LL^2(\Gamma)}.\]
Further, for a~one-dimensional Lipschitz manifold $\Gamma\subset\partial \Omega(\subseteq\R^2)$ with boundary, $\HH^{1/2}_0(\Gamma)$ is the space of elements of $\HH^{1/2}(\Gamma)$ which can, in $\HH^{1/2}$ be extended to zero outside $\Gamma$. It can be concluded from the trace theorem that 
trace operator has a~natural restriction to a~bounded and surjective operator $\gamma_{\Gamma}\colon \HH^{1}_{\partial\Omega\setminus\Gamma}(\Omega)\to \HH^{1/2}_0(\Gamma)$, i.e., $x\in \HH^{1}_{\partial\Omega\setminus\Gamma}(\Omega)$ is mapped to its
trace $x|_{\Gamma}$. Defining $\HH^{-1/2}(\Gamma)\coloneqq \HH^{1/2}(\Gamma)^*$, the {\em normal trace} of $x\in \HH(\Div,\Omega)$ at the relatively open set $\Gamma\subset\partial\Omega$ is well-defined by $w=\gamma_{N,\Gamma} x\in
\HH^{-1/2}_{\partial\Omega\setminus\Gamma}(\Gamma)$ with
\[
 \forall z\in \HH^{1}_{\partial\Omega\setminus\Gamma}(\Omega):\;\langle w,\gamma z\rangle_{\HH^{-1/2}(\Gamma),\HH^{1/2}(\Gamma)}=\langle \Div
 x,z\rangle_{\LL^2(\Omega)}+\langle x,\grad z\rangle_{\LL^2(\Omega;\C^2)}.\]
 We further set
 \[\HH_{\Gamma}(\Div,\Omega)=\setdef{z\in \HH(\Div,\Omega)}{\gamma_{N,\Gamma} z=0}.\]
\par\noindent\textbf{System nodes.} 
Next we introduce system nodes $S_1$, $S_2$ corresponding to each of the subsystems arising from the split of $\Omega$ into $\Omega_1$ and $\Omega_2$. 
In the following, we equip the spaces $X_i=\LL^2(\Omega_i)\times \LL^2(\Omega_i;\C^2)$ with the {\em energy norm}
\[\|\spvek{p_i}{q_i}\|^2_2\coloneqq \int_{\Omega_i}\rho^{-1}(\xi)p_i(\xi)^2+T(\xi)
q_i(\xi)^\top q_i(\xi)\,{\rm d}\xi,\quad i=1,2.\]
Note that, by the assumption that $\rho,T$ are positive-valued with $T\rho,\rho^{-1},T,T^{-1}\in \LL^{\infty}(\Omega)$, the energy norm is equivalent to the standard norm in $\LL^2(\Omega_i)\times \LL^2(\Omega_i;\C^2)$.\\
Then the weak formulations of \eqref{eq:wave1} and \eqref{eq:wave2} are, for 
\[x_i(t)\coloneqq\left.\sbvek{\rho z_t}{T^{-1}\grad z}\right|_{\Omega_i}\in X_i\coloneqq \sbvek{\LL^2(\Omega_i)}{\LL^2(\Omega_i;\C^2)},\quad i=1,2,\]
given by 
\[\sebvek{\dot{x}_1}{\mathfrak y}{y_1}=S_1 \sebvek{x_1}{\mathfrak u}{u_1},\quad \sbvek{\dot{x}_2}{y_2}=S_2 \sbvek{x_2}{u_2}.\]
Hereby, for $\mathfrak{U}=\mathfrak{U}_1=\HH^{-1/2}(\Gamma_0)$, $U_1=\HH^{-1/2}(\Gamma_{\rm int})$, the first system node is given by
\[S_1\colon \dom S_1\to \sbvek{X_1}{\sbvek{\mathfrak{U}}{U_1}}\]
with
\[\dom S_1=\setdef{
\sbvek{\sbvek{p_1}{q_1}}{\sbvek{\mathfrak u}{u_1}}\in \sbvek{X_1}{\sbvek{\mathfrak{U}}{U_1}}}{
\parbox[c]{6.3cm}{$\rho^{-1}p_1\in \HH^{1}_{\Gamma_{10}}(\Omega_1)\,\wedge\,Tq_1\in \HH(\Div,\Omega_1)$ $\;\;\wedge\, {\mathfrak u}=\gamma_{N,\Gamma_1}\rho^{-1}q_1\,\wedge\, u_1=\gamma_{N,\Gamma_{\rm int}}\rho^{-1}q_1$}
}\]
and
\[S_1\left[\begin{smallmatrix} p_1\\q_1\\ {\mathfrak u}\\u_1\end{smallmatrix}\right]=
\left[\begin{smallmatrix} \Div (Tq_1)-d\rho^{-1}p_1\\\grad (\rho^{-1}p_1)\\ \gamma_{\Gamma_1}(\rho^{-1} p_1)\\\gamma_{\Gamma_{\rm int}}(\rho^{-1} p_1)\end{smallmatrix}\right].\]
Likewise, $U_2=\HH^{1/2}(\Gamma_{\rm int})$, the second system node reads
\[S_2\colon \dom S_2\to \sbvek{X_2}{\sbvek{\mathfrak{U}}{U_2}}\]
with
\[\dom S_2=\setdef{
\sbvek{\sbvek{p_2}{q_2}}{\sbvek{\mathfrak u}{u_2}}\in \sbvek{X_2}{\sbvek{\mathfrak{U}}{U_2}}}{
\parbox[c]{6.3cm}{$\rho^{-1}p_2\in \HH^{1}_{\Gamma_{20}}(\Omega_2)\,\wedge\,Tq_2\in \HH(\Div,\Omega_1)$ $\;\;\,\wedge\, u_2=\gamma_{\Gamma_{\rm int}}\rho^{-1}p_2$}
}\]
and
\[S_2\left[\begin{smallmatrix} p_2\\q_2\\ {\mathfrak u}\\u_2\end{smallmatrix}\right]=
\left[\begin{smallmatrix} \Div (Tq_2)-d\rho^{-1}p_2\\\grad (\rho^{-1}p_2)\\0\\ \gamma_{N,\Gamma_{\rm int}}(Tq_2)\end{smallmatrix}\right].\]

\begin{remark}
    To be precise, the output spaces are not equal to the corresponding input spaces but dual to them (as hinted in the definition of the trace spaces). Hence, we can identify them with each other via isomorphic dual mappings and still use our results.
\end{remark}

Dissipativity follows from the definition of the trace operators, whereas the closedness claims follow from closedness of the divergence and gradient operators together with boundedness of the involved trace operators. Further, it can be shown that the main operators $A_1$, $A_2$ of $S_1$ and $S_2$, resp., fulfill $\dom A_1=\dom A_1^*$, $\dom A_2=\dom A_2^*$ and
\[A_i^*\left[\begin{smallmatrix} p_i\\q_i\end{smallmatrix}\right]=
\left[\begin{smallmatrix} -\Div (Tq_i)-d\rho^{-1}p_i\\-\grad (\rho^{-1}p_i)\end{smallmatrix}\right],\quad i=1,2.\]
This shows that they are both maximally dissipative, whence they generate a~strongly continuous semigroup on $X_i$.
Hence, given initial data and existence of a solution we can apply the presented splitting algorithm with the proven convergence properties. An additional advantage of this technique in this special case is that the decomposition into rectangles allows the usage of, say, spectral type methods to solve the two wave equations over rectangles separately, see, e.g., \cite{CohenWave}, \cite{Kopriva}.

\begin{remark}
Of course, nothing prevents to consider more than two rectangles and couple the corresponding wave equations analogously to the above. Moreover, the presented coupling is only one possible choice. If we forget about the interpretation of $u_i$ and $y_i$ as inputs and outputs and see the system as behavioral, choosing $u_1= u_2$ and $y_1=y_2$ suggests itself.
\end{remark}
\backmatter

\bibliography{references}


\begin{thebibliography}{36}
\ifx \bisbn   \undefined \def \bisbn  #1{ISBN #1}\fi
\ifx \binits  \undefined \def \binits#1{#1}\fi
\ifx \bauthor  \undefined \def \bauthor#1{#1}\fi
\ifx \batitle  \undefined \def \batitle#1{#1}\fi
\ifx \bjtitle  \undefined \def \bjtitle#1{#1}\fi
\ifx \bvolume  \undefined \def \bvolume#1{\textbf{#1}}\fi
\ifx \byear  \undefined \def \byear#1{#1}\fi
\ifx \bissue  \undefined \def \bissue#1{#1}\fi
\ifx \bfpage  \undefined \def \bfpage#1{#1}\fi
\ifx \blpage  \undefined \def \blpage #1{#1}\fi
\ifx \burl  \undefined \def \burl#1{\textsf{#1}}\fi
\ifx \doiurl  \undefined \def \doiurl#1{\url{https://doi.org/#1}}\fi
\ifx \betal  \undefined \def \betal{\textit{et al.}}\fi
\ifx \binstitute  \undefined \def \binstitute#1{#1}\fi
\ifx \binstitutionaled  \undefined \def \binstitutionaled#1{#1}\fi
\ifx \bctitle  \undefined \def \bctitle#1{#1}\fi
\ifx \beditor  \undefined \def \beditor#1{#1}\fi
\ifx \bpublisher  \undefined \def \bpublisher#1{#1}\fi
\ifx \bbtitle  \undefined \def \bbtitle#1{#1}\fi
\ifx \bedition  \undefined \def \bedition#1{#1}\fi
\ifx \bseriesno  \undefined \def \bseriesno#1{#1}\fi
\ifx \blocation  \undefined \def \blocation#1{#1}\fi
\ifx \bsertitle  \undefined \def \bsertitle#1{#1}\fi
\ifx \bsnm \undefined \def \bsnm#1{#1}\fi
\ifx \bsuffix \undefined \def \bsuffix#1{#1}\fi
\ifx \bparticle \undefined \def \bparticle#1{#1}\fi
\ifx \barticle \undefined \def \barticle#1{#1}\fi
\bibcommenthead
\ifx \bconfdate \undefined \def \bconfdate #1{#1}\fi
\ifx \botherref \undefined \def \botherref #1{#1}\fi
\ifx \url \undefined \def \url#1{\textsf{#1}}\fi
\ifx \bchapter \undefined \def \bchapter#1{#1}\fi
\ifx \bbook \undefined \def \bbook#1{#1}\fi
\ifx \bcomment \undefined \def \bcomment#1{#1}\fi
\ifx \oauthor \undefined \def \oauthor#1{#1}\fi
\ifx \citeauthoryear \undefined \def \citeauthoryear#1{#1}\fi
\ifx \endbibitem  \undefined \def \endbibitem {}\fi
\ifx \bconflocation  \undefined \def \bconflocation#1{#1}\fi
\ifx \arxivurl  \undefined \def \arxivurl#1{\textsf{#1}}\fi
\csname PreBibitemsHook\endcsname

\bibitem[\protect\citeauthoryear{Adams and Fournier}{2003}]{adams2003sobolev}
\begin{bbook}
\bauthor{\bsnm{Adams}, \binits{R.A.}},
\bauthor{\bsnm{Fournier}, \binits{J.J.}}:
\bbtitle{Sobolev Spaces}.
\bsertitle{Pure and Applied Mathematics},
vol. \bseriesno{140}.
\bpublisher{Academic Press},
\blocation{Amsterdam, San Diego, Oxford, London}
(\byear{2003}).
\bcomment{2nd edition}
\end{bbook}
\endbibitem

\bibitem[\protect\citeauthoryear{Barbu}{2010}]{Bar10}
\begin{bbook}
\bauthor{\bsnm{Barbu}, \binits{V.}}:
\bbtitle{Nonlinear Differential Equations of Monotone Types in Banach Spaces}.
\bsertitle{Springer Monographs in Mathematics}.
\bpublisher{Springer},
\blocation{New York}
(\byear{2010}).
\doiurl{10.1007/978-1-4419-5542-5}
\end{bbook}
\endbibitem

\bibitem[\protect\citeauthoryear{Bartel et~al.}{2022}]{BaGuJaRe22}
\begin{botherref}
\oauthor{\bsnm{Bartel}, \binits{A.}},
\oauthor{\bsnm{G\"unther}, \binits{M.}},
\oauthor{\bsnm{Jacob}, \binits{B.}},
\oauthor{\bsnm{Reis}, \binits{T.}}:
Operator splitting based dynamic iteration for linear port-Hamiltonian systems
(2022).
\url{https://arxiv.org/abs/2208.03574v1}
\end{botherref}
\endbibitem

\bibitem[\protect\citeauthoryear{B{\'a}tkai et~al.}{2016}]{BCsF2}
\begin{barticle}
\bauthor{\bsnm{B{\'a}tkai}, \binits{A.}},
\bauthor{\bsnm{Csom{\'o}s}, \binits{P.}},
\bauthor{\bsnm{Farkas}, \binits{B.}}:
\batitle{Operator splitting for dissipative delay equations}.
\bjtitle{Semigroup Forum}
\bvolume{95}(\bissue{2}),
\bfpage{345}--\blpage{365}
(\byear{2016}).
\doiurl{10.1007/s00233-016-9812-y}
\end{barticle}
\endbibitem

\bibitem[\protect\citeauthoryear{Bellen and Zennaro}{2013}]{BelZeBook}
\begin{bbook}
\bauthor{\bsnm{Bellen}, \binits{A.}},
\bauthor{\bsnm{Zennaro}, \binits{M.}}:
\bbtitle{Numerical Methods for Delay Differential Equations}.
\bsertitle{Numerical Mathematics and Scientific Computation}.
\bpublisher{Oxford University Press},
\blocation{Oxford}
(\byear{2013})
\end{bbook}
\endbibitem

\bibitem[\protect\citeauthoryear{Bellen et~al.}{2009}]{BelZeActa}
\begin{barticle}
\bauthor{\bsnm{Bellen}, \binits{A.}},
\bauthor{\bsnm{Maset}, \binits{S.}},
\bauthor{\bsnm{Zennaro}, \binits{M.}},
\bauthor{\bsnm{Guglielmi}, \binits{N.}}:
\batitle{Recent trends in the numerical solution of retarded functional
  differential equations}.
\bjtitle{Acta Numer.}
\bvolume{18},
\bfpage{1}--\blpage{110}
(\byear{2009}).
\doiurl{10.1017/S0962492906390010}
\end{barticle}
\endbibitem

\bibitem[\protect\citeauthoryear{Cohen}{2002}]{CohenWave}
\begin{bbook}
\bauthor{\bsnm{Cohen}, \binits{G.C.}}:
\bbtitle{Higher-order Numerical Methods for Transient Wave Equations}.
\bsertitle{Scientific Computation},
p. \bfpage{348}.
\bpublisher{Springer}, \blocation{???}
(\byear{2002}).
\doiurl{10.1007/978-3-662-04823-8}.
\bcomment{With a foreword by R. Glowinski}.
\burl{https://doi.org/10.1007/978-3-662-04823-8}
\end{bbook}
\endbibitem

\bibitem[\protect\citeauthoryear{{Cs}om{\'o}s et~al.}{2021}]{CsEF}
\begin{barticle}
\bauthor{\bsnm{{Cs}om{\'o}s}, \binits{P.}},
\bauthor{\bsnm{Ehrhardt}, \binits{M.}},
\bauthor{\bsnm{Farkas}, \binits{B.}}:
\batitle{Operator splitting for abstract cauchy problems with dynamical
  boundary condition}.
\bjtitle{Operators and Matrices}
\bvolume{15}(\bissue{3}),
\bfpage{903}--\blpage{935}
(\byear{2021})
\end{barticle}
\endbibitem

\bibitem[\protect\citeauthoryear{Csomós et~al.}{2023}]{CsFK}
\begin{barticle}
\bauthor{\bsnm{Csomós}, \binits{P.}},
\bauthor{\bsnm{Farkas}, \binits{B.}},
\bauthor{\bsnm{Kovács}, \binits{B.}}:
\batitle{{Error estimates for a splitting integrator for abstract semilinear
  boundary coupled systems}}.
\bjtitle{IMA Journal of Numerical Analysis}
(\byear{2023}).
\doiurl{10.1093/imanum/drac079}
\end{barticle}
\endbibitem

\bibitem[\protect\citeauthoryear{Curtain and Zwart}{2020}]{CuZw20}
\begin{bbook}
\bauthor{\bsnm{Curtain}, \binits{R.}},
\bauthor{\bsnm{Zwart}, \binits{H.}}:
\bbtitle{Introduction to Infinite-dimensional Systems Theory: {A} State Space
  Approach}.
\bsertitle{Texts in Applied Mathematics},
vol. \bseriesno{71}.
\bpublisher{Springer},
\blocation{New York}
(\byear{2020}).
\doiurl{10.1007/978-1-0716-0590-5}
\end{bbook}
\endbibitem

\bibitem[\protect\citeauthoryear{Dell{\textquotesingle}Oro
  et~al.}{2023}]{DePaSe22}
\begin{barticle}
\bauthor{\bsnm{Dell{\textquotesingle}Oro}, \binits{F.}},
\bauthor{\bsnm{Paunonen}, \binits{L.}},
\bauthor{\bsnm{Seifert}, \binits{D.}}:
\batitle{Optimal decay for a wave-heat system with
  {C}oleman{\textendash}{G}urtin thermal law}.
\bjtitle{Journal of Mathematical Analysis and Applications}
\bvolume{518}(\bissue{2}),
\bfpage{126706}
(\byear{2023}).
\doiurl{10.1016/j.jmaa.2022.126706}
\end{barticle}
\endbibitem

\bibitem[\protect\citeauthoryear{Engel and Nagel}{2000}]{EnNa00}
\begin{bbook}
\bauthor{\bsnm{Engel}, \binits{K.-J.}},
\bauthor{\bsnm{Nagel}, \binits{R.}}:
\bbtitle{One-parameter Semigroups for Linear Evolution Equations}
vol. \bseriesno{194}.
\bpublisher{Springer},
\blocation{New York}
(\byear{2000})
\end{bbook}
\endbibitem

\bibitem[\protect\citeauthoryear{Faou et~al.}{2015}]{Faou15}
\begin{barticle}
\bauthor{\bsnm{Faou}, \binits{E.}},
\bauthor{\bsnm{Ostermann}, \binits{A.}},
\bauthor{\bsnm{Schratz}, \binits{K.}}:
\batitle{Analysis of exponential splitting methods for inhomogenous parabolic
  equations}.
\bjtitle{IMA J. Numer. Anal.}
\bvolume{35},
\bfpage{161}--\blpage{178}
(\byear{2015})
\end{barticle}
\endbibitem

\bibitem[\protect\citeauthoryear{Geiser}{2011}]{GeiserBook}
\begin{bbook}
\bauthor{\bsnm{Geiser}, \binits{J.}}:
\bbtitle{Iterative Splitting Methods for Differential Equations}.
\bsertitle{Chapman \& Hall/CRC Numerical Analysis and Scientific Computing}.
\bpublisher{CRC Press},
\blocation{Boca Raton, FL}
(\byear{2011}).
\doiurl{10.1201/b10947}
\end{bbook}
\endbibitem

\bibitem[\protect\citeauthoryear{Grisvard}{1985}]{Gris85}
\begin{bbook}
\bauthor{\bsnm{Grisvard}, \binits{P.}}:
\bbtitle{Elliptic Problems in Nonsmooth Domains}.
\bsertitle{Monographs and Studies in Mathematics},
vol. \bseriesno{24}.
\bpublisher{Pitman Advanced Publishing Program},
\blocation{Boston, Londonm Melbourne}
(\byear{1985})
\end{bbook}
\endbibitem

\bibitem[\protect\citeauthoryear{G\"unther et~al.}{2021}]{BaGuJaRe21}
\begin{barticle}
\bauthor{\bsnm{G\"unther}, \binits{M.}},
\bauthor{\bsnm{Bartel}, \binits{A.}},
\bauthor{\bsnm{Jacob}, \binits{B.}},
\bauthor{\bsnm{Reis}, \binits{T.}}:
\batitle{Dynamic iteration schemes and port-{H}amiltonian formulation in
  coupled {DAE} circuit simulation}.
\bjtitle{International Journal of Circuit Theory and Applications}
\bvolume{49},
\bfpage{430}--\blpage{452}
(\byear{2021}).
\doiurl{10.1002/cta.2870}
\end{barticle}
\endbibitem

\bibitem[\protect\citeauthoryear{Hairer et~al.}{2010}]{HairerGeom}
\begin{bbook}
\bauthor{\bsnm{Hairer}, \binits{E.}},
\bauthor{\bsnm{Lubich}, \binits{C.}},
\bauthor{\bsnm{Wanner}, \binits{G.}}:
\bbtitle{Geometric Numerical Integration: Structure-preserving Algorithms for
  Ordinary Differential Equations}.
\bsertitle{Springer Series in Computational Mathematics},
vol. \bseriesno{31}.
\bpublisher{Springer},
\blocation{Heidelberg}
(\byear{2010})
\end{bbook}
\endbibitem

\bibitem[\protect\citeauthoryear{Hansen and Ostermann}{2009}]{HanO09}
\begin{barticle}
\bauthor{\bsnm{Hansen}, \binits{E.}},
\bauthor{\bsnm{Ostermann}, \binits{A.}}:
\batitle{Exponential splitting for unbounded operators}.
\bjtitle{Math. Comp.}
\bvolume{78}(\bissue{267}),
\bfpage{1485}--\blpage{1496}
(\byear{2009})
\end{barticle}
\endbibitem

\bibitem[\protect\citeauthoryear{Hansen and Henningsson}{2017}]{Hansen1}
\begin{barticle}
\bauthor{\bsnm{Hansen}, \binits{E.}},
\bauthor{\bsnm{Henningsson}, \binits{E.}}:
\batitle{Additive domain decomposition operator splittings---convergence
  analyses in a dissipative framework}.
\bjtitle{IMA J. Numer. Anal.}
\bvolume{37}(\bissue{3}),
\bfpage{1496}--\blpage{1519}
(\byear{2017}).
\doiurl{10.1093/imanum/drw043}
\end{barticle}
\endbibitem

\bibitem[\protect\citeauthoryear{Hansen et~al.}{2016}]{Hansen3}
\begin{barticle}
\bauthor{\bsnm{Hansen}, \binits{E.}},
\bauthor{\bsnm{Ostermann}, \binits{A.}},
\bauthor{\bsnm{Schratz}, \binits{K.}}:
\batitle{The error structure of the {D}ouglas-{R}achford splitting method for
  stiff linear problems}.
\bjtitle{J. Comput. Appl. Math.}
\bvolume{303},
\bfpage{140}--\blpage{145}
(\byear{2016}).
\doiurl{10.1016/j.cam.2016.02.037}
\end{barticle}
\endbibitem

\bibitem[\protect\citeauthoryear{Hochbruck and
  Ostermann}{2010}]{HochbruckOstermann}
\begin{barticle}
\bauthor{\bsnm{Hochbruck}, \binits{M.}},
\bauthor{\bsnm{Ostermann}, \binits{A.}}:
\batitle{Exponential integrators}.
\bjtitle{Acta Numer.}
\bvolume{19},
\bfpage{209}--\blpage{286}
(\byear{2010}).
\doiurl{10.1017/S0962492910000048}
\end{barticle}
\endbibitem

\bibitem[\protect\citeauthoryear{Hundsdorfer and Verwer}{1989}]{Hansen2}
\begin{barticle}
\bauthor{\bsnm{Hundsdorfer}, \binits{W.H.}},
\bauthor{\bsnm{Verwer}, \binits{J.G.}}:
\batitle{Stability and convergence of the {P}eaceman-{R}achford {ADI} method
  for initial-boundary value problems}.
\bjtitle{Math. Comp.}
\bvolume{53}(\bissue{187}),
\bfpage{81}--\blpage{101}
(\byear{1989}).
\doiurl{10.2307/2008350}
\end{barticle}
\endbibitem

\bibitem[\protect\citeauthoryear{Hundsdorfer and Verwer}{2003}]{HunVerBook}
\begin{bbook}
\bauthor{\bsnm{Hundsdorfer}, \binits{W.}},
\bauthor{\bsnm{Verwer}, \binits{J.}}:
\bbtitle{Numerical Solution of Time-dependent Advection-diffusion-reaction
  Equations}.
\bsertitle{Springer Series in Computational Mathematics},
vol. \bseriesno{33}.
\bpublisher{Springer},
\blocation{Berlin}
(\byear{2003}).
\doiurl{10.1007/978-3-662-09017-6}
\end{bbook}
\endbibitem

\bibitem[\protect\citeauthoryear{Jahnke and Lubich}{2000}]{JahnkeLubich}
\begin{barticle}
\bauthor{\bsnm{Jahnke}, \binits{T.}},
\bauthor{\bsnm{Lubich}, \binits{C.}}:
\batitle{Error bounds for exponential operator splittings}.
\bjtitle{BIT}
\bvolume{40}(\bissue{4}),
\bfpage{735}--\blpage{744}
(\byear{2000}).
\doiurl{10.1023/A:1022396519656}
\end{barticle}
\endbibitem

\bibitem[\protect\citeauthoryear{Kopriva}{2009}]{Kopriva}
\begin{bbook}
\bauthor{\bsnm{Kopriva}, \binits{D.A.}}:
\bbtitle{Implementing Spectral Methods for Partial Differential Equations}.
\bsertitle{Scientific Computation},
p. \bfpage{394}.
\bpublisher{Springer}, \blocation{???}
(\byear{2009}).
\doiurl{10.1007/978-90-481-2261-5}.
\bcomment{Algorithms for scientists and engineers}.
\burl{https://doi.org/10.1007/978-90-481-2261-5}
\end{bbook}
\endbibitem

\bibitem[\protect\citeauthoryear{Kurula and Zwart}{2015}]{KuZw15}
\begin{barticle}
\bauthor{\bsnm{Kurula}, \binits{M.}},
\bauthor{\bsnm{Zwart}, \binits{H.}}:
\batitle{Linear wave systems on n-d spatial domains}.
\bjtitle{International Journal of Control}
\bvolume{88}(\bissue{5}),
\bfpage{1063}--\blpage{1077}
(\byear{2015}).
\doiurl{10.1080/00207179.2014.993337}
\end{barticle}
\endbibitem

\bibitem[\protect\citeauthoryear{Lions and Mercier}{1979}]{Lions1979}
\begin{barticle}
\bauthor{\bsnm{Lions}, \binits{P.L.}},
\bauthor{\bsnm{Mercier}, \binits{B.}}:
\batitle{Splitting algorithms for the sum of two nonlinear operators}.
\bjtitle{SIAM Journal on Numerical Analysis}
\bvolume{16}(\bissue{6}),
\bfpage{964}--\blpage{979}
(\byear{1979}).
Accessed 2022-05-12
\end{barticle}
\endbibitem

\bibitem[\protect\citeauthoryear{Marchuk}{1990}]{Marchuk}
\begin{bchapter}
\bauthor{\bsnm{Marchuk}, \binits{G.I.}}:
\bctitle{Splitting and alternating direction methods}.
In: \bbtitle{Handbook of Numerical Analysis, {V}ol. {I}}.
\bsertitle{Handb. Numer. Anal., I},
pp. \bfpage{197}--\blpage{462}.
\bpublisher{Elsevier},
\blocation{Holland}
(\byear{1990})
\end{bchapter}
\endbibitem

\bibitem[\protect\citeauthoryear{McLachlan and Quispel}{2002}]{Quispel}
\begin{barticle}
\bauthor{\bsnm{McLachlan}, \binits{R.I.}},
\bauthor{\bsnm{Quispel}, \binits{G.R.W.}}:
\batitle{Splitting methods}.
\bjtitle{Acta Numer.}
\bvolume{11},
\bfpage{341}--\blpage{434}
(\byear{2002}).
\doiurl{10.1017/S0962492902000053}
\end{barticle}
\endbibitem

\bibitem[\protect\citeauthoryear{Peaceman and
  Rachford}{1955}]{PeacemanRachford}
\begin{barticle}
\bauthor{\bsnm{Peaceman}, \binits{D.W.}},
\bauthor{\bsnm{Rachford}, \binits{H.H.} \bsuffix{Jr.}}:
\batitle{The numerical solution of parabolic and elliptic differential
  equations}.
\bjtitle{J. Soc. Indust. Appl. Math.}
\bvolume{3},
\bfpage{28}--\blpage{41}
(\byear{1955})
\end{barticle}
\endbibitem

\bibitem[\protect\citeauthoryear{Staffans}{2002}]{St02a}
\begin{barticle}
\bauthor{\bsnm{Staffans}, \binits{O..}}:
\batitle{Passive and conservative continuous-time impedance and scattering
  systems. {I}.\ well-posed systems}.
\bjtitle{Math.\ Control Signals Systems}
\bvolume{15}(\bissue{4}),
\bfpage{291}--\blpage{315}
(\byear{2002})
\end{barticle}
\endbibitem

\bibitem[\protect\citeauthoryear{Staffans}{2005}]{St05}
\begin{bbook}
\bauthor{\bsnm{Staffans}, \binits{O.}}:
\bbtitle{Well-posed Linear Systems}.
\bsertitle{Encyclopedia of Mathematics and Its Applications},
vol. \bseriesno{103}.
\bpublisher{Cambridge University Press},
\blocation{Cambridge}
(\byear{2005})
\end{bbook}
\endbibitem

\bibitem[\protect\citeauthoryear{Strang}{1968}]{Strang68}
\begin{barticle}
\bauthor{\bsnm{Strang}, \binits{G.}}:
\batitle{On the construction and comparison of difference schemes}.
\bjtitle{SIAM J. Numer. Anal.}
\bvolume{5}(\bissue{3}),
\bfpage{506}--\blpage{517}
(\byear{1968}).
\doiurl{10.1137/0705041}
\end{barticle}
\endbibitem

\bibitem[\protect\citeauthoryear{Tartar}{2007}]{Tart07}
\begin{bbook}
\bauthor{\bsnm{Tartar}, \binits{L.}}:
\bbtitle{An Introduction to Sobolev Spaces and Interpolation Spaces}.
\bsertitle{Lecture Notes of the Unione Matematica Italiana}.
\bpublisher{Springer},
\blocation{Berlin, Heidelberg}
(\byear{2007})
\end{bbook}
\endbibitem

\bibitem[\protect\citeauthoryear{Tucsnak and Weiss}{2009}]{TuWe09}
\begin{bbook}
\bauthor{\bsnm{Tucsnak}, \binits{M.}},
\bauthor{\bsnm{Weiss}, \binits{G.}}:
\bbtitle{Observation and Control for Operator Semigroups}.
\bsertitle{Birkh\"auser Advanced Texts Basler Lehrb\"ucher}.
\bpublisher{Birkh\"auser},
\blocation{Basel}
(\byear{2009})
\end{bbook}
\endbibitem

\end{thebibliography}
\nocite{label}

\end{document}